\documentclass[11pt]{amsproc}

\usepackage{amsmath,amsfonts,amssymb,amsthm,amscd,mathdots,fullpage}
\usepackage{color}
\usepackage{tikz}
\RequirePackage[colorlinks,urlcolor=my-blue,linkcolor=my-red,citecolor=my-green]{hyperref}
\definecolor{my-blue}{rgb}{0.0,0.0,0.6}
\definecolor{my-red}{rgb}{0.9,0.0,0.0}
\definecolor{my-green}{rgb}{0.0,0.5,0.0}

\newtheorem{thm}{Theorem}[section]
\newtheorem{prop}[thm]{Proposition}
\newtheorem{cor}[thm]{Corollary}
\newtheorem{lem}[thm]{Lemma}

\newtheorem{theorem}{Theorem}[section]

\newtheorem{proposition}[theorem]{Proposition}
\newtheorem{corollary}[theorem]{Corollary}

\numberwithin{equation}{section}

\newcommand{\be}{\begin{equation}}  \newcommand{\ee}{\end{equation}}

\newcommand{\nn}{\nonumber}

\newcommand{\fl}[1]{\lfloor{#1}\rfloor}
\newcommand{\ce}[1]{\lceil{#1}\rceil}
\def\wt{\widetilde}   

\begin{document}

\def\C{{\mathbb C}}
\def\R{{\mathbb R}}
\def\F{{\mathcal F}}
\def\E{{\mathcal E}}

\def\Z{{\mathbb Z}}
\def\l{\lambda}
\def\tr{\triangle}
\def\varen{\varepsilon}
\def\did{\diamond}
\def\lot{\text{l.o.t.}}

\title[]{Geometric RSK correspondence, Whittaker functions and symmetrized random polymers}

\author{Neil O'Connell}
\address{Mathematics Institute, University of Warwick, Coventry CV4 7AL, UK}
\curraddr{}
\email{n.m.o-connell@warwick.ac.uk}
\thanks{}
\author{Timo Sepp\"al\"ainen}
\address{Department of Mathematics, University of Wisconsin-Madison, Madison, WI 53706-1388, USA}
\curraddr{}
\email{seppalai@math.wisc.edu}
\thanks{}
\author{Nikos Zygouras}
\address{Department of Statistics, University of Warwick, Coventry CV4 7AL, UK}
\curraddr{}
\email{n.zygouras@warwick.ac.uk}
\thanks{}



\date{}

\dedicatory{}

\begin{abstract}
We show that the geometric lifting of the RSK correspondence introduced by A.N. Kirillov (2001)
is volume preserving with respect to a natural product measure on its domain, and that the integrand in 
Givental's integral formula for $GL(n,\R)$-Whittaker functions arises naturally in this context. 
Apart from providing further evidence that Whittaker functions are the natural analogue of Schur 
polynomials in this setting, our results also provide a new `combinatorial' framework for the study of random polymers.
When the input matrix consists of random inverse gamma distributed weights, the probability 
distribution of a polymer partition function constructed from these weights can be written down 
explicitly in terms of Whittaker functions.   
Next we restrict the geometric RSK mapping to symmetric matrices and show that the volume 
preserving property continues to hold. We determine the probability law of the polymer partition 
function with inverse gamma weights that are constrained to be symmetric about the main diagonal, 
with an additional factor on the main diagonal.  
The third combinatorial mapping studied is a variant of the geometric RSK mapping for triangular arrays, which is 
again showed to be volume preserving.  This leads to a formula for the probability distribution of a polymer model 
whose paths are constrained to stay below the diagonal.  We also show that the analogues of the Cauchy-Littlewood 
identity in the setting of this paper are equivalent to a collection of Whittaker integral identities conjectured by Bump 
(1989) and  Bump and Friedberg (1990) and proved by Stade (2001, 2002).  Our approach leads to new `combinatorial' 
proofs and generalizations of these identities, with some restrictions on the parameters.  
\end{abstract}

\maketitle

\tableofcontents

\section{Introduction}

The Robinson-Schensted-Knuth (RSK) correspondence is a combinatorial mapping which plays an
important role in the theory of Young tableaux, symmetric functions and representation theory~\cite{fulton,stanley}.  
It is deeply connected with Schur functions and provides a combinatorial framework for understanding the Cauchy-Littlewood identity and Schur measures on integer partitions.  It is also the basic structure which lies behind the solvability of a particular family of combinatorial models in probability and statistical physics including longest increasing subsequence problems, directed last passage percolation in 1+1 dimensions and the totally asymmetric simple exclusion process, see for example~\cite{AD,BDJ,KJ,OkInfWedge}.

The RSK map is defined on matrices with non-negative integer coefficients and
can be described by expressions in the max-plus semi-ring.  This was extended to matrices with
real entries by Berenstein and Kirillov~\cite{bki}.  Replacing these expressions by their analogues in the usual algebra, 
A.N. Kirillov~\cite{ki} introduced a geometric lifting of the Berenstein-Kirillov correspondence which he called 
the `tropical RSK correspondence', in honour of M.-P. Sch\"utzenberger (1920--1996).
However, for many readers nowadays the word `tropical' indicates just the opposite, 
so to avoid confusion we will refer to Kirillov's construction as the {\em geometric} RSK (gRSK) correspondence, 
as in the theory of {\em geometric crystals}~\cite{bk,bk2}, which is closely related.

The geometric RSK correspondence is a birational mapping from $(\R_{>0})^{n\times m}$ onto itself. 
It was introduced by Kirillov~\cite{ki} for square matrices ($n=m$) and generalized to the rectangular 
setting by Noumi and Yamada~\cite{ny}.   In the paper~\cite{cosz} it was shown 
that there is a fundamental connection between the gRSK correspondence and $GL(n,\R)$-Whittaker 
functions, analogous to the well-known connection between the RSK correspondence and Schur
functions.  In particular, it is explained there that the analogue of the Cauchy-Littlewood identity in the
setting of gRSK can be seen as a generalization of a Whittaker integral identity which was
originally conjectured by Bump~\cite{bump} and later proved by Stade~\cite{stade}.  The connection
to Whittaker functions gives rise to a natural family of measures (Whittaker measures) which play
a similar role in this setting to Schur measures on integer partitions.  It also has applications 
to random polymers.  In the paper~\cite{cosz}, an
explicit integral formula is obtained for the Laplace transform of the law of the partition function 
associated with a random directed polymer model on the two-dimensional lattice with log-gamma 
weights introduced in~\cite{s}.  For related recent developments, see~\cite{noc,bc,bcr}.

In the present work, we first provide further insight into the results of~\cite{cosz} by showing:

(a) the gRSK mapping is volume preserving with respect to the product measure 
$\prod_{ij} dx_{ij}/x_{ij}$ on $(\R_{>0})^{n\times m}$, and 

(b) the integrand in
Givental's integral formula for $GL(n,\R)$-Whittaker functions~\cite{g,jk} arises naturally through the 
application of the gRSK map (see Theorem~\ref{bi} below).  

The volume preserving property can be seen as a consequence of a new description of the gRSK map
as a composition of {\em local moves} which we introduce in this paper.  This description is 
a re-formulation of the geometric row-insertion algorithm introduced by Noumi and Yamada in~\cite{ny}.  
Combining (a) and (b) gives a direct `combinatorial' proof of Stade's identity (with some restrictions on the parameters)
analogous to the bijective proof of the Cauchy-Littlewood identity via the classical RSK correspondence
(see, for example, Fulton~\cite[\S 4.3]{fulton}).

The second aim of this paper is to initiate a program of understanding the gRSK
mapping in the presence of symmetry constraints in much the same spirit as the work of Baik
and Rains~\cite{br,br1,b} on longest increasing subsequence and last passage percolation
problems.  Here we consider one
particular symmetry, namely the restriction of gRSK to symmetric matrices.  We show
that the volume preserving property continues to hold in this setting and deduce the analogue
of the Whittaker measure.  The corresponding Whittaker integral identity (Corollary \ref{nwid})
involves only a single Whittaker function, and turns out to be 
equivalent to a formula for a certain Mellin transform of the 
$GL(n,\R)$-Whittaker function which was conjectured by Bump and Friedberg~\cite{bf} and proved by 
Stade~\cite{stade-jams}, again with some restrictions on the parameters.  
We also consider a degeneration of this model in which the diagonal entries of the input matrix
vanish and the gRSK map rescales to a new version of gRSK defined on triangles.  This
model has a surprising and non-trivial connection to the symmetric case.

As an application we determine the law of the partition function of a family of random polymer models 
with log-gamma weights that are constrained to be symmetric about the main diagonal.
We also consider a degeneration of this model in which the polymer paths are constrained to stay 
below the diagonal.  This  can be seen as a discrete version of the continuum random polymer above a hard wall,
which appeared recently in the physics literature \cite{gld}.  Formally, our results yield integral formulae 
for the Laplace transforms of these laws which can be used as a starting point for further asymptotic development.
Similar integral formulae obtained in \cite{cosz} for the polymer model without symmetry were used in~\cite{bcr} 
to prove Tracy-Widom GUE asymptotics for the law of the partition function.  The polymer models we
consider also give rise to a positive temperature version of the interpolating ensembles of Baik and Rains~\cite{br,b}.
In the KPZ scaling limit they should correspond to the KPZ equation on the half-line with mixed boundary
conditions at zero and narrow wedge initial condition.

The outline of the paper is as follows.  

$\bullet$ In the next section we give some background on Whittaker functions,
introduce a generalization of these functions and explain how these functions can be regarded as
generating functions for {\em patterns}.  This interpretation can be seen as a generalization of Givental's 
integral formula~\cite{g,jk,gklo} and is analogous to the combinatorial interpretation of Schur functions as 
generating functions for semistandard Young tableaux.  

$\bullet$ In Section~\ref{gRSKsec} we give a new description of the gRSK map as a composition of local moves
(based on Noumi and Yamada's dynamical description of gRSK) and use this to establish several
basic results.  In particular, we show that the gRSK mapping is volume-preserving with respect
to a natural product measure on $(\R_{>0})^{n\times m}$ and establish a fundamental
identity (Theorem~\ref{bi}) which provides an elementary explanation of the appearance of Whittaker 
functions in this setting.  This gives further insight into earlier results from~\cite{cosz} and yields a new 
proof and generalization of two of Stade's Whittaker integral identities (Theorems \ref{st1} and \ref{st2}).  

$\bullet$ 
In Section~\ref{nvo} we explain the relationship between the local-moves description of gRSK and the geometric
row-insertion algorithm of Noumi and Yamada~\cite{ny}.  

$\bullet$ 
In Section~\ref{sym-sec} we consider the restriction of gRSK
to symmetric matrices.  We show that the volume preserving property continues to hold in this setting 
and deduce several consequences, including a new proof (with some restriction on the parameters)
of the Whittaker integral identity (Theorem \ref{st3}) involving a single Whittaker function due to 
Stade~\cite{stade-jams}.  

$\bullet$ 
In Section \ref{sec:wall} we introduce gRSK  for triangular arrays.  Again we prove a
fundamental identity and the volume preserving property, 
and deduce the probability distribution of the shape vector of the output
array under inverse gamma distributed initial weights.   
The polymer version of the problem describes paths restricted to lie below a hard wall.  

$\bullet$ 
In Section~\ref{wid} we explain how the results of this paper relate to some of the Whittaker integral 
identities which have appeared previously in the automorphic forms literature.  

$\bullet$ 
In Section \ref{sec:trop} we explain how the Berenstein-Kirillov extension of the RSK correspondence
can be recovered by taking a limit (tropicalization).  In statistical physics terminology this is a zero-temperature
limit that takes polymer partition functions to last-passage percolation values.    
By analogy with Section \ref{gRSKsec}, we give a description of the Berenstein-Kirillov mapping
in terms of local moves which shows that this map is also volume preserving.  Under exponentially 
distributed weights the probability distribution of the shape vector of the resulting pair of Gelfand-Tsetlin 
patterns is given by a non-central Laguerre ensemble.  This connection to random matrix theory
has had important applications to last passage percolation models~\cite{KJ,br,fr,bp,dw}.

\bigskip

\noindent {\bf Acknowledgements.}  Many thanks to Jinho Baik, Ivan Corwin, Anatol Kirillov,
Eric Rains and Eric Stade for helpful discussions.
NO'C is partially supported by EPSRC grant EP/I014829/1. TS is partially supported by National Science 
Foundation grants DMS-1003651 and DMS-1306777 and by the Wisconsin Alumni Research Foundation. NZ is supported by a Marie Curie International Reintegration Grant within the 7th European Community Framework Programme, IRG-246809.

\section{Whittaker functions and patterns} \label{wp}

We begin by defining the following Baxter $Q$-type operators, as in~\cite{gklo,glo}.
For $\lambda\in\C$, $x,y\in(\R_{>0})^n$, define
$$Q^{n}_\lambda(x,y)=\left(\prod_{i=1}^n\frac{y_i}{x_i}\right)^\lambda
\exp\left( -\sum_{i=1}^n\frac{y_i}{x_i} -\sum_{i=1}^{n-1}\frac{x_{i+1}}{y_i} \right) .$$
For $\lambda\in\C$, $x\in(\R_{>0})^n$ and $y\in(\R_{>0})^{n-1}$, define
$$Q^{n,n-1}_\lambda(x,y)=\left(\frac{\prod_{i=1}^{n-1} y_i}{\prod_{i=1}^n x_i}\right)^\lambda
\exp\left( -\sum_{i=1}^{n-1} \frac{y_i}{x_i} -\sum_{i=1}^{n-1}\frac{x_{i+1}}{y_i} \right) .$$
We regard these as integral operators: for suitable test functions,
$$Q^n_\lambda f (x) = \int_{(\R_{>0})^n} Q^{n}_\lambda(x,y) f(y) \prod_{i=1}^n \frac{dy_i}{y_i},$$
$$Q^{n,n-1}_\lambda f (x) = \int_{(\R_{>0})^{n-1}} Q^{n,n-1}_\lambda(x,y) f(y) \prod_{i=1}^{n-1} \frac{dy_i}{y_i}.$$
Define $\Psi^n_\lambda(x)$, $\lambda\in\C^n$, $x\in(\R_{>0})^n$
 recursively by setting $\Psi^1_{\lambda}(x)=x^{-\lambda}$ and, for $n\ge 2$,
\begin{equation}\label{wdef}
\Psi^n_{\lambda_1,\ldots,\lambda_n} = Q^{n,n-1}_{\lambda_n} \Psi^{n-1}_{\lambda_1,\ldots,\lambda_{n-1}}.
\end{equation}
We note here, for later reference, some identities which follow easily from the definitions.
For $a> 0$ we have 
\begin{equation}\label{a}
\Psi^n_\alpha(ax)=a^{-\sum_i\alpha_i} \Psi^n_\alpha(x).
\end{equation}
If $\alpha_i'=\alpha_i+c$ for some $c\in\C$, then
\begin{equation}\label{shift}
\Psi^n_{\alpha'}(x)=\left(\prod_i x_i^{-c}\right) \Psi^n_\alpha(x).
\end{equation}
Finally, if we set $x_i'=1/x_{n-i+1}$, then
\begin{equation}\label{inv}
\Psi^n_{\lambda}(x)=\Psi^n_{-\lambda}(x').
\end{equation}
The functions $\Psi^n_\lambda$ are $GL(n,\R)$-Whittaker functions. The above definition is essentially
Givental's integral formula~\cite{g,jk}  (see also~\cite{gklo,glo}).  These functions were
first introduced by Jacquet~\cite{ja}.  They play an important role in the theory of 
automorphic forms~\cite{bump1,bump,bf,stade,stade-jams,dg,is}
and the quantum Toda lattice~\cite{k,sts,g,jk,gklo,glo,kl}.  In the latter literature they arise as eigenfunctions
of the open quantum Toda chain with $n$ particles with Hamiltonian given by
$$H=-\sum_i \frac{\partial^2}{\partial x_i^2} + 2 \sum_{i=1}^{n-1} e^{x_{i+1}-x_i}.$$
If we define $\psi^n_\lambda(x)=\Psi^n_{-\lambda}(z)$, where $x_i=\log z_i$ for $i=1,\ldots,n$,
then $$H\psi^n_\lambda=-\left(\sum_i\lambda_i^2\right) \psi^n_\lambda.$$
See, for example,~\cite{gklo} for more details.

In the automorphic forms literature the standard `normalisation' is slightly different.
In particular, in the notation of the paper~\cite{is}, we have the relation, for $n\ge 2$:
\begin{equation}\label{rel}
\Psi^n_{-\lambda}(x)=\left(\prod_i x_i\right)^{(1/n)\sum_i\lambda_i} \left(\prod_{j=1}^{n-1} y_j^{-j(n-j)/2} \right) W_{n,a}(y),
\end{equation}
where $a_k=\lambda_k-(1/n)\sum_i\lambda_i$ for $k=1,\ldots,n$ and $\pi y_j =\sqrt{x_{n-j+1}/x_{n-j}},$
for $j=1,\ldots,n-1$.  This is easily verified by comparing the recursion (\ref{wdef}) with a similar
recursion obtained by Ishii and Stade~\cite{is} for the functions $W_{n,a}(y)$, and using the
elementary relation (\ref{shift}).
Indeed, first note that, by (\ref{shift}), we only need to check this for $\lambda=a$, that is,
when $\sum_i\lambda_i=0$. In the case $n=2$ we have, writing $a=(a,-a)$ and $y_1=y$,
$$W_{2,a}(y)=2\sqrt{y} K_{2a}(2\pi y)=\sqrt{y}\Psi^{2}_{a}(x_1,x_2)$$
where $\pi y=\sqrt{x_2/x_1}$ and $K_\nu$ is the Macdonald function
$$K_\nu(z)=\frac12\int_0^\infty t^{\nu-1} e^{-\frac{z}{2}(t+t^{-1})} dt .$$
For $n\ge 3$, in~\cite{is} it is shown that 
\begin{align*}
W_{n,a}(y)=\prod_{j=1}^{n-1} y_j^{(n-j)/2+2a_1(n-j)/(n-1)} \int_{(\R_{>0})^{n-1} }
e^{-\pi \sum_{j=1}^{n-1} y_j^2 u_j +1/u_j} \\
\prod_{j=1}^{n-1} u_j^{(n-2j)/4+na_1/(n-1)}
W_{n-1,b}\left( y_2\sqrt{\frac{u_2}{u_1}},\ldots, y_{n-1}\sqrt{\frac{du_{n-1}}{u_{n-2}}}\right)
\frac{du_1}{u_1}\ldots\frac{du_{n-1}}{u_{n-1}} ,
\end{align*}
where $$b=\left( a_2+\frac{a_1}{n-1},\ldots,a_n+\frac{a_1}{n-1}\right).$$
Making the change of variables 
$$\pi y_j =\sqrt{\frac{x_{n-j+1}}{x_{n-j}} },\qquad \frac\pi{u_j} = \frac{x_{n-j}}{z_{n-j}}, \qquad
\pi y_j^2 u_j = \frac{z_{n-j+1}}{x_{n-j}} ,$$
for $j=1,\ldots,n-1$, and using (\ref{rel}) above, we see that this is equivalent to the recursion
$$\Psi^n_{-a}(x)=\int_{(\R_{>0})^{n-1} } Q_{-a_1}^{n,n-1}(x,z) \Psi^{n-1}_{-a_2,\ldots,-a_n} (z) 
\frac{dz_1}{z_1}\ldots\frac{dz_{n-1}}{z_{n-1}} ,$$
which agrees with (\ref{wdef}) above.

We will also consider the following generalization of the functions $\Psi^n_\lambda$.
For $\lambda\in\C^n$, $x\in(\R_{>0})^n$ and $s\in\C$,
define 
\begin{equation}\label{wgendef1}
\Psi^n_{\lambda;s}(x)=e^{-s/x_n}\Psi_\lambda^n(x);
\end{equation}
for $\lambda\in\C^{n+k}$, $k\ge 1$, and $\Re s>0$, define
\begin{equation}\label{wgendef2}
\Psi^n_{\lambda;s} = 
Q^n_{\lambda_{n+k}} Q^n_{\lambda_{n+k-1}} \ldots Q^n_{\lambda_{n+1}} \Psi^n_{\lambda_1,\ldots,\lambda_n;s} .
\end{equation}
It is straightforward to see that $\Psi^n_{\lambda;s}(x)$ is well-defined, as an absolutely convergent integral, 
for each $x\in\R^n$.  The functions $\Psi^n_{\lambda;s}$
can be regarded as generating functions for `patterns', as we shall now explain.

Let $x\in (\R_{>0})^n$.  We define a {\em pattern} $P$ with {\em shape} $\mbox{sh } P=x$ 
and {\em height} $h\ge n$ to be an array of
positive real numbers
\begin{eqnarray}\label{Ppattern}
P=\begin{array}{cccccccccc}
&&&z_{11}&&&&&\\
&&z_{22}&&z_{21}&&&&\\
&\iddots&&&&\ddots&&&\\
z_{nn}&&&\ldots&&&z_{n1}&&\\
&\ddots&&&&&&\ddots&\\
&&z_{hn}&&&\ldots&&&z_{h1}
\end{array}
\end{eqnarray}
with bottom row $z_{h\cdot} = x$.  The range of indices is 
$$L(n,h)=\{(i,j):\ 1\le i\le h,\ 1\le j\le i\wedge n\}.$$
If $h=n$ then $P$ is a {\em triangle} in the sense of Kirillov~\cite{ki}.
Fix a pattern $P$ as above.  Set $\rho_0=1$ and, for $1\le i\le h$,
$\rho_i=\prod_{j=1}^{i\wedge n} z_{ij}$ and $\tau_i=\rho_i/\rho_{i-1}$.
We shall refer to $\tau$ as the {\em type} of $P$ and write $\tau=\mbox{type }P$.
For $\alpha\in\C^h$ define
\begin{eqnarray}\label{Ptype}
P^\alpha = \prod_{i=1}^h \tau_i^{\alpha_i}.
\end{eqnarray}
For $s\in\C$, define 
\begin{equation}\label{fsp}
\F_s(P)=\frac{s}{z_{nn}} + \sum_{(i,j)\in L(n,h)} \frac{z_{i-1,j}+z_{i+1,j+1}}{z_{ij}}
\end{equation}
with the convention that $z_{ij}=0$ if $(i,j)\notin L(n,h)$.
Denote by $\Pi^h(x)$ the set of patterns with shape $x$ and height $h$.
Then, for $\lambda\in\C^h$ and $\Re s>0$ (this condition is only required if $h>n$)
\begin{equation}\label{gf}
\Psi^n_{\lambda;s}(x)=\int_{\Pi^h(x)} P^{-\lambda} e^{-\F_s(P)} dP
\end{equation}
where
$$dP = \prod_{(i,j)\in L(n,h-1)} \frac{dz_{ij}}{z_{ij}} .$$
This formula is just a re-writing of the above definition (\ref{wgendef2}) of $\Psi^n_{\lambda;s}$.

We remark that, although it is not obvious from the above definition, the function $\Psi^n_\lambda$
is invariant under permutations of the indices $\lambda_1,\ldots,\lambda_n$~\cite{kl,glo}.  In fact, the same 
is true of the function $\Psi^n_{\lambda;s}$, where $\lambda\in\C^{n+k}$, $k\ge 1$ and $\Re s>0$.  That is,
$\Psi^n_{\lambda;s}$ is invariant under permutations of the indices $\lambda_1,\ldots,\lambda_{n+k}$.  This 
follows from the definition (\ref{wgendef2}), using the relation 
\begin{equation}\label{Qc}
Q^n_aR^n_sQ^n_b=Q^n_bR^n_sQ^n_a,
\end{equation}
where $R_s$ denotes multiplication by the function $e^{-s/x_n}$, and the invariance
of $\Psi^n_{\lambda_1,\ldots,\lambda_n}$ under permutations of $\lambda_1,\ldots,\lambda_n$.
The relation (\ref{Qc}) is a straightforward extension of the commutativity property 
$Q^n_aQ^n_b=Q^n_bQ^n_a$ obtained in \cite[Theorem 2.3]{glo}.

There is a Plancherel theorem for the Whittaker functions~\cite{w,sts,kl}, 
which states that the integral transform
$$\hat f(\lambda) = \int_{(\R_{>0})^n} f(x) \Psi^n_\lambda(x) \prod_{i=1}^n \frac{dx_i}{x_i}$$
defines an isometry from $L_2((\R_{>0})^n, \prod_{i=1}^n dx_i/x_i)$ onto 
$L^{sym}_2(\iota\R^n,s_n(\lambda)d\lambda)$, where $L_2^{sym}$ is the space of $L_2$
functions which are symmetric in their variables, $\iota=\sqrt{-1}$ and 
$$s_n(\lambda)=\frac1{(2\pi\iota)^n n!} \prod_{i\ne j} \Gamma(\lambda_i-\lambda_j)^{-1},$$
is the {\em Sklyanin} measure.

\section{Geometric RSK correspondence}   \label{gRSKsec}

The geometric RSK (gRSK) correspondence is a bijective mapping $$T:(\R_{>0})^{n\times m}\to (\R_{>0})^{n\times m}.$$
It is also birational in the sense that both $T$ and its inverse are rational maps.
It was introduced by Kirillov~\cite{ki} as a geometric lifting of the Berenstein-Kirillov correspondence and further
studied by Noumi and Yamada~\cite{ny}.  We will define it here via a sequence of `local moves' on matrix elements.  
This is essentially a reformulation of the row-insertion procedure introduced in~\cite{ny}, 
as will be explained in Section~\ref{nvo} below.

For each $2\le i\le n$ and $2\le j\le m$ define 
a mapping $l_{ij}$ which takes as input a matrix $X=(x_{ij})\in(\R_{>0})^{n\times m}$ and replaces the submatrix
$$ \begin{pmatrix} x_{i-1,j-1}& x_{i-1,j}\\ x_{i,j-1}& x_{ij}\end{pmatrix}$$
of $X$ by its image under the map
\begin{equation}\label{abcd}
\begin{pmatrix} a& b\\ c& d\end{pmatrix} \qquad \mapsto \qquad 
\begin{pmatrix} bc/(ab+ac) & b\\ c& d(b+c) \end{pmatrix},
\end{equation}
and leaves the other elements unchanged.  For $2\le i\le n$ and $2\le j\le m$, 
define $l_{i1}$ to be the mapping that replaces
the element $x_{i1}$ by $x_{i-1,1}x_{i1}$ and $l_{1j}$ to be the mapping that replaces the element $x_{1j}$ by 
$x_{1,j-1}x_{1j}$.  For convenience  define $l_{11}$ to be the identity map.
For $1\le i\le n$ and $1\le j\le m$, set
$$\pi^j_i=l_{ij}\circ\cdots\circ l_{i1},$$
and, for $1\le i\le n$,
\be R_i = \begin{cases} \pi_1^{m-i+1}\circ\cdots\circ\pi^m_i & i\le m\\
\pi^1_{i-m+1} \circ\cdots\circ\pi^m_i & i\ge m . \end{cases}  \label{defR}\ee
The mapping $T$ is defined by
\be T=R_n\circ\cdots\circ R_1. \label{defT}\ee
For example, suppose $n=m=2$.  Then 
$$R_1=\pi^2_1=l_{12}\circ l_{11}=l_{12},\qquad R_2=\pi^1_1\circ\pi^2_2=l_{11}\circ l_{22}\circ l_{21}=l_{22}\circ l_{21}$$
and so
$$T=R_2\circ R_1=l_{22}\circ l_{21}\circ l_{12}.$$
Here is an illustration:
\begin{align*}
T: \begin{pmatrix} a&b\\c&d\end{pmatrix} \stackrel{l_{12}}{\longmapsto}
\begin{pmatrix} a&ab\\c&d\end{pmatrix} \stackrel{l_{21}}{\longmapsto}
\begin{pmatrix} a&ab\\ac&d\end{pmatrix} \stackrel{l_{22}}{\longmapsto}
\begin{pmatrix} bc/(b+c)&ab\\ac&ad(b+c)\end{pmatrix} .
\end{align*}

Note that each $l_{ij}$ is birational.  For example, the inverse of the map (\ref{abcd}) is given by
$$\begin{pmatrix} a& b\\ c& d\end{pmatrix} \qquad \mapsto \qquad 
\begin{pmatrix} bc/(ab+ac) & b\\ c& d/(b+c) \end{pmatrix}.  $$
The birational property of $T$ can thus be seen directly from the above definition.

Each matrix $X\in (\R_{>0})^{n\times m}$ can be identified with a pair of patterns $(P,Q)$
with respective heights $m$ and $n$, and common shape
$$\mbox{sh }P=\mbox{sh Q} =(x_{nm},x_{n-1,m-1},\ldots,x_{n-p+1,m-p+1}),$$ where $p=n\wedge m$,
as illustrated in the following example:
$$ X = \begin{array}{cccc} & & x_{31} & \\ & x_{21} & & x_{32} \\ x_{11} & & x_{22} & \\ & x_{12} & & \end{array} $$
\bigskip
$$P = \begin{array}{ccc}  & x_{31} & \\  x_{21} & & x_{32} \end{array}, \qquad
Q = \begin{array}{cccc}  & x_{12} & & \\ x_{11} & & x_{22} & \\ & x_{21} & & x_{32}  \end{array} $$
\bigskip
$$\mbox{sh }P=\mbox{sh Q} = (x_{32},x_{21}).$$
In the following, we will simply write $X=(P,Q)$ to indicate that $X$ is identified with the pair $(P,Q)$.

The mappings $R_i$ defined above can also be written as
$$R_i=\rho^i_m\circ\cdots\circ\rho^i_2\circ\rho^i_1$$
where
\begin{eqnarray}\label{rodef}
\rho^i_j=\begin{cases} l_{1,j-i+1}\circ\cdots\circ l_{i-1,j-1}\circ l_{ij} & i\le j\\
l_{i-j+1,1}\circ\cdots\circ l_{i-1,j-1}\circ l_{ij} & i\ge j.\end{cases}
\end{eqnarray}
Here we are just using the obvious fact that $l_{ij}\circ l_{i'j'}=l_{i'j'}\circ l_{ij}$ whenever $|i-i'|+|j-j'|>2$.
This representation is closely related to the Bender-Knuth transformations, as we shall now explain.
For each $1\le i\le n$ and $1\le j\le m$, denote by $b_{ij}$ the map on $(\R_{>0})^{n\times m}$ which
takes a matrix $X=(x_{qr})$ and replaces the entry $x_{ij}$ by
\begin{equation}\label{bk1}
x'_{ij} = \frac1{x_{ij}} (x_{i,j-1}+x_{i-1,j}) \left( \frac1{x_{i+1,j}}+\frac1{x_{i,j+1}}\right)^{-1},
\end{equation}
leaving the other entries unchanged, with the conventions that $x_{0j}=x_{i0}=0$, 
$x_{n+1,j}=x_{i,m+1}=\infty$ for $1< i<n$ and $1< j<m$, but 
$x_{10}+x_{01}=x_{n+1,m}^{-1}+x_{n,m+1}^{-1}=1$.
Denote by $r_{j}$ the map which replaces the entry $x_{nj}$ by
$x_{n,j+1}/x_{nj}$ if $j<m$ and $1/x_{nm}$ if $j=m$, leaving the other entries unchanged.  
For $j\le m$, define
\begin{equation}\label{bk}
h_j=\begin{cases} b_{n-j+1,1}\circ\cdots\circ b_{n-1,j-1}\circ b_{nj} & j\le n\\
b_{1,j-n+1}\circ\cdots\circ b_{n-1,j-1}\circ b_{nj} & j\ge n.\end{cases}
\end{equation}
It is straightforward from the definitions to see that $\rho^n_j=h_j\circ r_{j}$.  Now, observe
that if $X=(P,Q)$, then for each $j< m$, $h_j(X)=(t_j(P),Q)$ where $t_j$ is defined by this relation.
It is easy to see that the mappings $b_{ij}$, $h_j$ and $t_j$ are all involutions.

In the case $n=m$, the mappings $t_1,\ldots,t_{n-1}$ are the analogues of the Bender-Knuth
transformations in this setting, as discussed in~\cite{ki}. In this case, if we define, for $i<n$,
\begin{equation}\label{sch}
q_{i}=t_1\circ(t_2\circ t_1)\circ\cdots\circ(t_i\circ\cdots\circ t_1),
\end{equation}
then, as explained in~\cite{ki}, 
the involutions $s_i=q_i\circ t_1 \circ q_i$, $i<n$, satisfy the braid
relations $(s_i s_{i+1})^3=Id.$, 
and hence define an action of $S_n$ on the set of triangles of height $n$.
The mapping $q_{n-1}$ is the analogue of Sch\"utzenberger's involution in this setting.

An immediate consequence of the above re-formulation of gRSK is the following volume preserving property.
Denote the input matrix by $W=(w_{ij})\in (\R_{>0})^{n\times m}$ and the output matrix by 
$T=T(W)=(t_{ij})\in (\R_{>0})^{n\times m}$.
\begin{thm}\label{vp} The gRSK mapping in logarithmic variables 
$$(\log w_{ij},\ 1\le i\le n, 1\le j\le m)\mapsto (\log t_{ij},\ 1\le i\le n, 1\le j\le m)$$
has Jacobian $\pm 1$. \end{thm}
\begin{proof}
It is easy to see that the Jacobians of the mappings $l_{ij}$ in logarithmic variables are all $\pm 1$.  
This follows from the fact that the mappings
$$(\log a, \log b) \mapsto (\log a, \log a + \log b)$$
$$(\log a,\log b,\log c,\log d)\mapsto (\log(bc/(ab+ac)),\log b,\log c,\log(db+dc))$$ 
each have Jacobian $\pm 1$.  The result follows from the definition~(\ref{defT}) of $T$.
\end{proof}
We remark that, by a similar argument it can be seen that the involutions $q_i,\ i<n$, on the set of triangles 
of height $n$, all have Jacobian $\pm 1$ in logarithmic variables.

We recall here some basic properties of the gRSK map $T$, which are either obvious from the definitions
or proved in the papers~\cite{k,ny}.  Suppose $W\in (\R_{>0})^{n\times m}$ and $T=T(W)=(P,Q)$.
If we define row and column products $R_i=\prod_j w_{ij}$ and 
$C_j=\prod_i w_{ij}$, then $\mbox{type }Q=R$ and $\mbox{type }P=C$.
Note that this implies, for $\lambda\in\C^m$ and $\nu\in\C^n$, 
\begin{equation}\label{pq}
\prod_{ij} w_{ij}^{-\nu_i-\lambda_j} = \prod_i R_i^{-\nu_i} \prod_j C_j^{-\lambda_j} = P^{-\lambda} Q^{-\nu}.
\end{equation}
Also, the following symmetries hold:
\begin{itemize}
\item[] $T(W^t)=T(W)^t$;
\item[]  $T(W)=(P,Q) \iff T(W^t)=(Q,P)$; 
\item[] $W$ is symmetric $\iff$ $T$ is symmetric $\iff P=Q$;
\item[] $W$ is symmetric across the anti-diagonal $\iff Q=q_{n-1}(P)$.
\end{itemize}
The connection to directed polymers is via the following formula for $t_{nm}$:
$$t_{nm}=\sum_{\pi\in\Pi_{n,m}} \prod_{(i,j)\in\pi} w_{ij},$$
where $\Pi_{n,m}$ is the set of directed nearest-neighbor lattice paths in $\Z^2$ from $(1,1)$ to $(n,m)$, 
that is, the set of paths
$\pi=\{\pi(1),\pi(2),\ldots,\pi(n+m-1)\}$ such that $\pi(1)=(1,1)$, $\pi(n+m-1)=(n,m)$
and $\pi(k+1)-\pi(k)\in\{(1,0),(0,1)\}$ for $1\le k<n+m-1$.
We shall refer to the variable $t_{nm}$ as the {\em polymer partition function}.
In this context it is natural to refer to the $w_{ij}$ as {\em weights} and $W$ as the {\em weight matrix}.
In fact, the remaining entries of $T=(P,Q)$ can also be expressed in terms of
similar partition functions, as follows.  For $1\le k\le m$ and $1\le r\le n\wedge k$,
\be t_{n-r+1,k-r+1}\dotsm t_{n-1,k-1} t_{nk} 
= \sum_{(\pi_1,\ldots,\pi_r)\in\Pi^{(r)}_{n,k}} \prod_{(i,j)\in \pi_1\cup\cdots\cup\pi_r} w_{ij},
\label{path1}\ee
where $\Pi^{(r)}_{n,k}$ denotes the set of $r$-tuples of non-intersecting
directed nearest-neighbor lattice paths $\pi_1,\ldots,\pi_r$
starting at positions $(1,1),(1,2),\ldots,(1,r)$ and ending at positions $(n,k-r+1),\ldots,(n,k-1),(n,k)$. (See Figure \ref{fig5}. When we use the path representation we draw the weight matrix in Cartesian coordinates.) 
This determines the entries of $P$.  The entries of $Q$ are given by similar formulae using
$T(W^t)=(Q,P)$.  We note here the following identity, which follows  from the
above lattice path representation for $T$: setting $p=n\wedge m$, we have
\begin{equation}\label{t11}
\sum_{i=1}^p \frac1{w_{i,p-i+1}} = \frac1{t_{11}}.
\end{equation}
To see this if $n\le m$, take the ratio of \eqref{path1} for $\Pi^{(n-1)}_{n,n}$ and $\Pi^{(n)}_{n,n}$. 
In the opposite case apply the same to $W^t$. 

\setlength{\unitlength}{1.3pt}

\begin{figure}[t]
 \begin{center}
\begin{tikzpicture}[scale=.7]
\draw[help lines] (1,1) grid (7,8);
\draw[very thick] (1,1)--(2,1)--(3,1)--(3,2)--(4,2)--(4,3)--(5,3)--(6,3)--(7,3)--(7,4);
\draw[very thick] (1,2)--(2,2)--(2,3)--(2,4)--(3,4)--(4,4)--(5,4)--(5,5)--(6,5)--(7,5);
\draw[very thick] (1,3)--(1,4)--(1,5)--(2,5)--(3,5)--(3,6)--(4,6)--(5,6)--(6,6)--(7,6);
\draw [fill] (1,1) circle [radius=0.1];
\draw [fill] (1,2) circle [radius=0.1];
\draw [fill] (1,3) circle [radius=0.1];
\draw [fill] (7,4) circle [radius=0.1];
\draw [fill] (7,5) circle [radius=0.1];
\draw [fill] (7,6) circle [radius=0.1];

\node at (0,1) {\small{$(1,1)$}};
\node at (8,6) {\small{$(n,k)$}};
\node at (8,8) {\small{$(n,m)$}};
\end{tikzpicture}

\end{center}  
\caption{ \small Three paths $(\pi_1, \pi_2, \pi_3)$  of a particular  3-tuple in  $\Pi^{(3)}_{n,k}$ in an $n\times m$ weight matrix. 
Note that the picture is in Cartesian coordinates.  The  paths start at the lower left at $(1,1)$, $(1,2)$ and 
$(1,3)$ and end at the upper right at $(n,k-2)$, $(n,k-1)$, $(n,k)$. }  \label{fig5}
\end{figure}
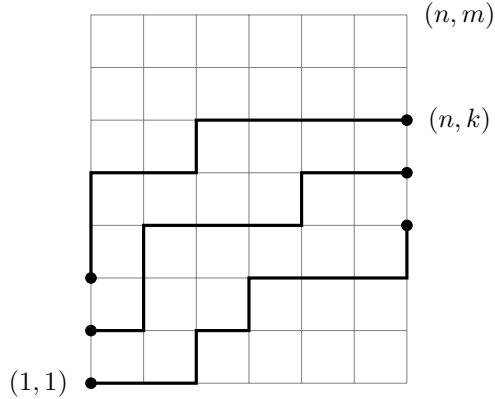

Now, for $X\in (\R_{>0})^{n\times m}$ and $s\in\C$, define
\begin{equation}\label{e}
\E_s(X)=\frac{s}{x_{11}}+\sum_{ij} \frac{x_{i-1,j}+x_{i,j-1}}{x_{ij}},
\end{equation}
where 
the summation is over $1\le i\le n,\ 1\le j\le m$ with the conventions $x_{ij}=0$ for $i=0$ or $j=0$.
Note that, if $X=(P,Q)$, then
$$\E_s(X)=\begin{cases} \F_0(P)+\F_s(Q) & n\ge m\\ \F_s(P)+\F_0(Q) & n\le m  \end{cases}$$
where $\F_s$ is defined by (\ref{fsp}).  An important property of the maps $b_{ij}$ defined by
(\ref{bk1}) above is that they preserve the quantity $\E_0(X)$, that is, $\E_0\circ b_{ij}=\E_0$.
To see this, recall that the map $b_{ij}$ takes a matrix $X=(x_{qr})$ and replaces the entry 
$x_{ij}$ by
$$x'_{ij} = \frac1{x_{ij}} (x_{i,j-1}+x_{i-1,j}) \left( \frac1{x_{i+1,j}}+\frac1{x_{i,j+1}}\right)^{-1},$$
leaving the other entries unchanged, with the conventions that $x_{0j}=x_{i0}=0$, 
$x_{n+1,j}=x_{i,m+1}=\infty$ for $1<i<n$ and $1<j<m$, and 
$x_{10}+x_{01}=x_{n+1,m}^{-1}+x_{n,m+1}^{-1}=1$.  
It is readily verified that
\begin{equation}\label{b-inv}
\frac{x'_{i,j-1}+x'_{i-1,j}}{x'_{ij}}+\frac{x'_{ij}}{x'_{i+1,j}}+\frac{x'_{ij}}{x'_{i,j+1}}
= \frac{x_{i,j-1}+x_{i-1,j}}{x_{ij}}+\frac{x_{ij}}{x_{i+1,j}}+\frac{x_{ij}}{x_{i,j+1}}
\end{equation}
with the conventions that $x_{0j}=x_{i0}=x'_{0j}=x'_{i0}=0$ and 
$x_{n+1,j}=x_{i,m+1}=x'_{n+1,j}=x'_{i,m+1}=\infty$ for each $i$ and $j$.
This implies $\E_0(b_{ij}(X))=\E_0(X)$.
We remark that, in particular, this implies $\E_0\circ h_j=\E_0$, 
$\F_0\circ t_j=\F_0$ for all $j< m$ and, in the case $m=n$, $\F_0\circ q_{n-1}=\F_0$, 
where $q_{n-1}$ is the geometric analogue of Sch\"utzenberger's involution defined by (\ref{sch}).

The cornerstone of the present paper is the following identity which, combined with Theorem~\ref{vp},
explains the appearance of $GL(n,\R)$-Whittaker functions in the context of geometric RSK.
\begin{thm}\label{bi}   Let $W\in (\R_{>0})^{n\times m}$, $T=T(W)$ and $s\in\C$.  Then
$$\sum_{i=1}^p \frac{s}{w_{i,p-i+1}} + \sum {}^{'} \frac{1}{w_{ij}} = \E_s(T),$$
where $p=n\wedge m$ and $\sum^{'}$ denotes the sum over $1\le i\le n,\ 1\le j\le m$ such that $j\ne p-i+1$.  
\end{thm}
\begin{proof}
From the identity (\ref{t11}), we can assume without loss of generality that $s=1$.
We will prove the theorem by induction on $n$ and $m$.  The statement is immediate in the case $n=m=1$.
Write $R_i=R_i^{n,m}$, $T=T^{n,m}$ and $\E_s^{n,m}$ for the mappings defined above.
Recall that $T^{m,n}(W^t)=[T^{n,m}(W)]^t$, for any values of $m$ and $n$.
It therefore suffices to show that the proposition holds for $T^{n,m}$,
assuming that $n\ge m$ and that the proposition holds for $T^{n-1,m}$.

Write $W_{n-1,m}=(w_{ij},\ 1\le i\le n-1,\ 1\le j\le m)$, $S=T^{n-1,m}(W_{n-1,m})$ and $T=T^{n,m}(W)$.
Then
$$T=R_n^{n,m} \begin{pmatrix} S \\ w_{n1}\ \ldots\ w_{nm} \end{pmatrix} ,$$
and we are required to show that
$$\E_1^{n,m}(T)=\E_1^{n-1,m}(S)+\sum_{j=1}^m \frac1{w_{nj}}.$$
Now,
$$R^{n,m}_n = \rho^n_m\circ\cdots\circ\rho^n_2\circ\rho^n_1$$
where 
$$\rho^n_k= h_k \circ r_{k}=b_{n-k+1,1}\circ\cdots\circ b_{nk}\circ r_{k}.$$
Set $$T^{(0)}=\begin{pmatrix} S \\ w_{n1}\ \ldots\ w_{nm} \end{pmatrix} $$
and, for $k=1,\ldots,m$, $$T^{(k)}=\rho^n_k\circ\cdots\circ\rho^n_2\circ\rho^n_1 (T^{(0)}).$$
For $X\in (\R_{>0})^{n\times m}$ and $0\le k\le m$, define
$$\E^{n,m;k}(X)=\frac{1}{x_{11}}+ \sum_{ij}^{(k)} \frac{x_{i-1,j}+x_{i,j-1}}{x_{ij}}+\sum_{j=k+1}^m \frac1{x_{nj}},$$
where $X=(x_{ij})$ and the first summation is over pairs of indices $(i,j)$ such that either $1\le i< n$ 
and $1\le j\le m$ or $i=n$ and $1\le j\le k$, with the conventions $x_{ij}=0$ for $i=0$ or $j=0$.
Note that
$$\E^{n,m;0}(T^{(0)})=\E_1^{n-1,m}(T)+\sum_{j=1}^m \frac1{w_{nj}},\qquad 
\E^{n,m;m}(T^{(m)})=\E_1^{n,m}(T).$$
We will show that 
\begin{equation}\label{bbi}
\E^{n,m;k}\circ\rho^n_k=\E^{n,m;k-1}
\end{equation} 
for each $k=1,\ldots,m$.
Note that this implies $$\E^{n,m;k}(T^{(k)})=\E^{n,m;k-1}(T^{(k-1)})$$ for each $k=1,\ldots,m$, 
and the statement of the theorem follows.  

Let $X=(x_{ij})\in (\R_{>0})^{n\times m}$ and write 
$$X'=(x'_{ij})=\rho^n_k(X)= h_k\circ r_k (X).$$
Note that $x'_{ij}=x_{ij}$ for all $(i,j)$ except $(n-q+1,k-q+1)$, $1\le q\le k$.
Applying $b_{nk}\circ r_k$ gives the relation
$$ \frac{x'_{n,k-1}+x'_{n-1,k}}{x'_{nk}}= \frac1{x_{nk}}.$$
The next three relations follow from the invariance of $\E_0$ under the $b_{ij}$
mappings as discussed earlier, see (\ref{b-inv}).
If $(i,j)=(n-q+1,k-q+1)$ for some $1<q<k$, then
$$\frac{x'_{i,j-1}+x'_{i-1,j}}{x'_{ij}}+\frac{x'_{ij}}{x'_{i+1,j}}+\frac{x'_{ij}}{x'_{i,j+1}}
= \frac{x_{i,j-1}+x_{i-1,j}}{x_{ij}}+\frac{x_{ij}}{x_{i+1,j}}+\frac{x_{ij}}{x_{i,j+1}}.$$
If $k<n$, then
$$\frac{x'_{n-k,1}}{x'_{n-k+1,1}}+\frac{x'_{n-k+1,1}}{x'_{n-k+2,1}}+\frac{x'_{n-k+1,1}}{x'_{n-k+1,2}}
= \frac{x_{n-k,1}}{x_{n-k+1,1}}+\frac{x_{n-k+1,1}}{x_{n-k+2,1}}+\frac{x_{n-k+1,1}}{x_{n-k+1,2}};$$
If $k=n$ (this can only occur if $m=n$), then
$$\frac{1}{x'_{11}}+\frac{x'_{11}}{x'_{21}}+\frac{x'_{11}}{x'_{12}}
= \frac{1}{x_{11}}+\frac{x_{11}}{x_{21}}+\frac{x_{11}}{x_{12}}.$$
It follows that $\E^{n,m;k}(X')=\E^{n,m;k-1}(X)$, as required.
\end{proof}

Let $s>0$ and consider the measure on input matrices $(w_{ij})$ defined by
$$\nu_{\hat\theta,\theta;s}(dw) = \prod_{ij} w_{ij}^{-\hat\theta_i-\theta_j}
\exp\left( -  \sum_{i=1}^p \frac{s}{w_{i,p-i+1}} - \sum {}^{'} \frac1{w_{ij}} \right) \prod_{ij} \frac{dw_{ij}}{w_{ij}},$$
where $\hat\theta_i+\theta_j>0$ for each $i$ and $j$.
Note that
$$\int_{(\R_{>0})^{n\times m} }\nu_{\hat\theta,\theta;s}(dw)
 = s^{-\sum_{i=1}^p (\hat\theta_i+\theta_{p-i+1})} \prod_{ij}\Gamma(\hat \theta_i+\theta_j) .$$
Suppose $W\in (\R_{>0})^{n\times m}$ and $T=T(W)=(P,Q)$.  
Define a mapping $\sigma: (\R_{>0})^{n\times m} \to (\R_{>0})^p$ by
\begin{equation}\label{sigma}
\sigma(W)=\mbox{sh }P=\mbox{sh }Q =(t_{nm},t_{n-1,m-1},\ldots,t_{n-p+1,m-p+1}),
\end{equation}
where $p=n\wedge m$.
The next two corollaries are essentially a re-formulation of two of the main results in~\cite{cosz}.

\begin{cor}\label{pft}
The push-forward of the measure $\nu_{\hat\theta,\theta;s}$ under the geometric RSK map $T$ is given by
$$\nu_{\hat\theta,\theta;s}\circ T^{-1} (dt) = P^{-\theta} Q^{-\hat\theta} e^{ - \E_s(T) }
 \prod_{ij} \frac{dt_{ij}}{t_{ij}} .$$
\end{cor}
\begin{proof}
This follows immediately from Theorems~\ref{vp}, \ref{bi} and the relation (\ref{pq}).
\end{proof}

\begin{cor}\label{pf} 
The push-forward of $\nu_{\hat\theta,\theta;s}$ under $\sigma$ is given by
$$\nu_{\hat\theta,\theta;s}\circ\sigma^{-1} (dx) = 
\begin{cases}
\Psi^p_\theta(x) \Psi^p_{\hat\theta;s}(x)\prod_{i=1}^p \frac{dx_i}{x_i} &n\ge m \\
\Psi^p_{\theta;s}(x) \Psi^p_{\hat\theta}(x)\prod_{i=1}^p \frac{dx_i}{x_i} & n\le m.\end{cases} $$
\end{cor}
\begin{proof}
This follows from Corollary \ref{pft} and the integral formula (\ref{gf}) for $\Psi^p_{\lambda;s}$.
\end{proof}

We also obtain from Theorems \ref{vp} and \ref{bi} the following integral identity.
This is the analogue of the Cauchy-Littlewood identity in this setting.
\begin{cor} \label{cor-bs-gen} Suppose $s>0$, $\lambda\in\C^m$ and $\nu\in\C^n$, where $n\ge m$ and
$\Re (\lambda_i+\nu_j)>0$ for all $i$ and $j$.  Then
\begin{equation}\label{bs-gen} 
\int_{(\R_{>0})^m } \Psi^m_{\nu;s}(x) \Psi^m_\lambda(x) \prod_{i=1}^m \frac{dx_i}{x_i}
= s^{-\sum_{i=1}^m (\nu_i+\lambda_i)} \prod_{ij}\Gamma(\nu_i+\lambda_j) .
\end{equation} 
\end{cor}
\begin{proof} From the definitions (\ref{gf}), (\ref{e}), the identity (\ref{pq}), Theorems \ref{vp} and \ref{bi}, 
and Fubini's theorem,
\begin{align*}
& s^{-\sum_{i=1}^m (\nu_i+\lambda_i)}  \prod_{ij}\Gamma(\nu_i+\lambda_j) \\
&= \int_{(\R_{>0})^{n\times m} } 
\prod_{ij} w_{ij}^{-\nu_i-\lambda_j-1}
\exp\left( -  \sum_{i=1}^m \frac{s}{w_{i,m-i+1}} - \sum {}^{'} \frac1{w_{ij}} \right) \prod_{ij} dw_{ij} \\
&= \int_{(\R_{>0})^{n\times m} } 
P^{-\lambda} Q^{-\nu} e^{ - \E_s(T) } \prod_{ij} \frac{dt_{ij}}{t_{ij}} \\
&= \int_{(\R_{>0})^m }\left( \int_{\Pi^n(x)} Q^{-\nu} e^{-\F_s(Q)} dQ\right)
\left( \int_{\Pi^m(x)} P^{-\lambda} e^{-\F_0(P)} dP\right) \prod_{i=1}^m \frac{dx_i}{x_i} \\
&= \int_{(\R_{>0})^m }  \Psi^m_{\nu;s}(x) \Psi^m_\lambda(x) \prod_{i=1}^m \frac{dx_i}{x_i} ,\\
\end{align*}
as required.
\end{proof}

When $m=n-1$ this is equivalent to an integral identity which was conjectured by Bump~\cite{bump}
and proved by Stade~\cite[Theorem 3.4]{stade-jams}, see Theorem~\ref{thm-stade2-psi} below.  
We note that in this case, the identity is proved in~\cite{stade-jams} without assuming the condition 
$\Re (\lambda_i+\nu_j)>0$ for all $i$ and $j$.  In this case, the integral is associated with archimedean 
$L$-factors of automorphic $L$-functions on $GL(n-1,\R)\times GL(n,\R)$. 

When $n=m$, (\ref{bs-gen}) becomes: 
\begin{cor}\label{bs1}
Suppose $s>0$ and $\lambda,\nu\in\C^n$, where $\Re (\lambda_i+\nu_j)>0$ for all $i$ and $j$.  Then
$$
\int_{(\R_{>0})^n } e^{-s/x_n} \Psi^n_\nu(x) \Psi^n_{\lambda}(x)\prod_{i=1}^n \frac{dx_i}{x_i}
= s^{-\sum_{i=1}^n (\nu_i+\lambda_i)} \prod_{ij}\Gamma(\nu_i+\lambda_j) .
$$
\end{cor}
Using \eqref{inv}, this is equivalent to the following integral identity for $GL(n,\R)$-Whittaker functions, 
due to Stade~\cite{stade}, see Theorem~\ref{thm-stade-psi} below.
\begin{cor}[Stade] \label{stade}
 Suppose $r>0$ and $\lambda,\nu\in\C^n$, where $\Re (\lambda_i+\nu_j)>0$ for all $i$ and $j$.  Then
$$\int_{(\R_{>0})^n } e^{-r x_1} \Psi^n_{-\nu}(x) \Psi^n_{-\lambda}(x)\prod_{i=1}^n \frac{dx_i}{x_i}
= r^{-\sum_{i=1}^n (\nu_i+\lambda_i)} \prod_{ij}\Gamma(\nu_i+\lambda_j).$$
\end{cor} 
Again, we note that this identity is proved in~\cite{stade} without assuming the condition
$\Re (\lambda_i+\nu_j)>0$ for all $i$ and $j$.  In this case, the integral is associated, via the Rankin-Selberg 
method, with Archimedean $L$-factors of automorphic $L$-functions on $GL(n,\R)\times GL(n,\R)$. 
\begin{cor}\label{L2}
Suppose $s>0$ and $\nu\in\C^n$ with $\Re\nu_i>0$ for each $i$.  Then, for each $m\le n$, the function 
$\Psi^m_{\nu;s}$ is in $L_2((\R_{>0})^m, \prod_{i=1}^m dx_i/x_i)$, and the function
$e^{-s x_1} \Psi^n_{-\nu}(x)$ is in $L_2((\R_{>0})^n, \prod_{i=1}^n dx_i/x_i)$.
\end{cor}
\begin{proof}
The first claim follows from Corollary \ref{cor-bs-gen} and the Plancherel theorem, as follows.
We first note that, under the above hypotheses,
$$\hat {\Psi}^m_{\nu;s} (\lambda) =
s^{-\sum_{i=1}^m (\nu_i+\lambda_i)} \prod_{ij}\Gamma(\nu_i+\lambda_j) 
\in L_2(\iota\R^m, s_m(\lambda)d\lambda).$$
This is easily verified using Stirling's approximation
$$\lim_{b\to\infty}|\Gamma(a+\iota b)| e^{\frac{\pi}{2}|b|}|b|^{\frac{1}{2}-a}=\sqrt{2\pi},
\qquad a,b \in \mathbb{R} .$$
Now, suppose $f\in L_2((\R_{>0})^m, \prod_{i=1}^m dx_i/x_i)$ such that $\hat f$ is continuous
and compactly supported on $\iota\R^m$.  By the Plancherel theorem, such functions are dense
in $L_2((\R_{>0})^m, \prod_{i=1}^m dx_i/x_i)$ and, moreover, satisfy
\begin{equation}\label{invers}
f(x)=\int_{\iota\R^m} \hat f(\lambda) \Psi^m_{\lambda}(x) s_m(\lambda) d\lambda
\end{equation}
almost everwhere.  Indeed, for any $g\in L_2((\R_{>0})^m, \prod_{i=1}^m dx_i/x_i)$ which is continuous
and compactly supported we have, by Fubini's theorem,
\begin{align*}
\int_{(\R_{>0})^m} \left( \int_{\iota\R^m} \hat f(\lambda) \Psi^m_{\lambda}(x) s_m(\lambda) d\lambda
\right) \overline{g(x)} \prod_{i=1}^m \frac{dx_i}{x_i}
&=\int_{\iota\R^m} \hat f(\lambda) \overline{\hat{g}(\lambda)} s_m(\lambda) d\lambda\\
&= \int_{(\R_{>0})^m} f(x) \overline{g(x)} \prod_{i=1}^m \frac{dx_i}{x_i} .
\end{align*}
This implies (\ref{invers}).  Now, by Corollary \ref{cor-bs-gen},
$$\int_{(\R_{>0})^m} |\Psi^m_{\nu;s}(x)\Psi^m_{\lambda}(x)| \prod_{i=1}^m \frac{dx_i}{x_i}
\le \int_{(\R_{>0})^m} \Psi^m_{\Re\nu;s}(x)\Psi^m_0(x) \prod_{i=1}^m \frac{dx_i}{x_i} <\infty .$$
It follows that, for $f\in L_2((\R_{>0})^m, \prod_{i=1}^m dx_i/x_i)$ such that $\hat f$ is continuous
and compactly supported on $\iota\R^m$, the integral
$$\int_{(\R_{>0})^m}\int_{\iota\R^m}  \Psi^m_{\nu;s}(x) \overline{\hat f(\lambda)} \Psi^m_{\lambda}(x) s_m(\lambda) d\lambda \frac{dx_i}{x_i}$$
is absolutely convergent, and so, by Fubini's theorem,
\begin{align*}
\int_{(\R_{>0})^m} \Psi^m_{\nu;s}(x) \overline{f(x)} \prod_{i=1}^m\frac{dx_i}{x_i}
&= \int_{(\R_{>0})^m} \Psi^m_{\nu;s}(x) \left( \int_{\iota\R^m} \overline{\hat f(\lambda)} \Psi^m_{\lambda}(x) s_m(\lambda) d\lambda
\right) \prod_{i=1}^m \frac{dx_i}{x_i}\\
&=\int_{\iota\R^m}  \hat\Psi^m_{\nu;s}(\lambda) \overline{\hat f(\lambda)} s_m(\lambda) d\lambda .
\end{align*}
Hence, using the Cauchy-Schwarz inequality, 
\begin{align*}
& \left\vert \int_{(\R_{>0})^m}  \Psi^m_{\nu;s}(x) \overline{f(x)} \prod_{i=1}^m \frac{dx_i}{x_i} \right\vert =
\left\vert \int_{\iota\R^m}  \hat\Psi^m_{\nu;s}(\lambda) \overline{\hat f(\lambda)} s_m(\lambda) d\lambda \right\vert \\
&\le \left( \int_{\iota\R^m} | \hat\Psi^m_{\nu;s}(\lambda)|^2 s_m(\lambda) d\lambda\right)^{1/2}
\left( \int_{\iota\R^m} |\hat f(\lambda)|^2 s_m(\lambda) d\lambda \right)^{1/2} .
\end{align*}
This proves the first claim.  The second claim follows from the first, letting $m=n$ and using (\ref{inv}).
\end{proof}

Consider the probability measure on input matrices $W$ defined by 
$$\tilde \nu_{\hat\theta,\theta;s}(dw)= Z_{\hat\theta,\theta;s}^{-1} \nu_{\hat\theta,\theta;s}(dw)$$
where
$$Z_{\hat\theta,\theta;s} = s^{-\sum_{i=1}^p (\hat\theta_i+\theta_i)} \prod_{ij}\Gamma(\hat \theta_i+\theta_j).$$
The following result was obtained in~\cite{cosz}.
\begin{cor}
Suppose $\hat\theta_i+\theta_j>0$ for each $i$ and $j$, and  (w.l.o.g.) that $n\ge m$, $\theta_i<0$ for each $i$ and $\hat\theta_j>0$ for each $j$.
Then, the Laplace transform of the law $\tilde \nu_{\hat\theta,\theta;s}\circ t_{nm}^{-1}$
of the polymer partition function $t_{nm}$ under $\tilde \nu_{\hat\theta,\theta;s}$ 
is given by
$$\int e^{-r t_{nm}} \tilde \nu_{\hat\theta,\theta;s}(dw)
= \int_{\iota\R^m} (rs)^{\sum_{i=1}^m (\theta_i-\lambda_i)} \prod_{ij}\Gamma(\lambda_i-\theta_j)
\prod_{ij}\frac{\Gamma(\hat \theta_i+\lambda_j)}{\Gamma(\hat \theta_i+\theta_j)}
 s_n(\lambda) d\lambda .$$
\end{cor}
\begin{proof}
By Corollary~\ref{pf}, 
$$ \int  e^{-r t_{nm}} \tilde \nu_{\hat\theta,\theta;s}(dw)
= Z_{\hat\theta,\theta;s}^{-1} \int_{(\R_{>0})^m } 
e^{-rx_1} \Psi^m_\theta(x) \Psi^m_{\hat\theta;s}(x)\prod_{i=1}^m \frac{dx_i}{x_i} .$$
By Corollary~\ref{L2}, the functions $e^{-rx_1} \Psi^m_\theta(x)$ and $\Psi^m_{\hat\theta;s}(x)$
are in the space $L_2((\R_{>0})^m, \prod_{i=1}^m dx_i/x_i)$.  
The result follows, by Corollaries~\ref{bs1}, ~\ref{stade} and the Plancherel theorem.
\end{proof}

\section{Equivalence of old and new description of geometric RSK}\label{nvo}

We explain here the equivalence of the Noumi-Yamada row insertion
construction \cite{ny}  and the  definition of geometric RSK 
given in Section \ref{gRSKsec}. 
   The input weight matrix $(w_{ij})$ is $n\times m$, where $m$ is fixed and 
$n$ represents time. 
After $n$ time steps the Noumi-Yamada process gives two patterns $P=\{z_{k\ell}\}$ and 
$Q=\{z'_{ij}\}$.  
  $P$ has height $m$, $Q$ has height $n$,   and their
  common shape vector $z_{m\,\centerdot}=z'_{n\,\centerdot}$ is of length $p=m\wedge n$.  
The rows of $Q$ indexed by $s=1,\dotsc,n$ 
 from top to bottom are the successive shape vectors  (bottom rows) 
$z_{m\,\centerdot}(s)= (z_{m,\ell}(s))_{1\le\ell\le m\wedge s}$ of the temporal 
evolution $\{z(s): 1\le s\le n\}$ of the $P$ pattern.    Thus as in classic RSK
the $Q$ pattern serves as a recording pattern.  

 The Noumi-Yamada process begins with an empty pattern at time $n=0$.  
 Then the following step is repeated for $n=1,2,3,\dotsc$.  

\medskip

{\sl Noumi-Yamada construction for time step $n-1\to n$.}  Let $z=z(n-1)$
denote the $P$ pattern obtained after $n-1$ steps.  
Insertion of   row $w_{n\,\centerdot}$ of weights into $z$ 
  transforms 
  $z$ into $\check z=z(n)$ as follows.  
 
(i)  If $n\ge m+1$ (in other words, the triangle was filled by time $n-1$), then 
\be\begin{aligned}\label{NYalg1}
a_{k,1}&=w_{n,k}   &\textrm{for } 1\le k\leq m\\
a_{k+1,\ell+1}&=a_{k+1,\ell} \frac{z_{k+1,\ell} \check z_{k,\ell}}{\check z_{k+1,\ell} z_{k,\ell}}
&\textrm{for } 1\le \ell\leq k<m\\
\check z_{k,\ell}&= a_{k,\ell}(z_{k,\ell}+\check z_{k-1,\ell})&\textrm{for } 1\le \ell<k\leq m\\
\check z_{k,k}&= a_{k,k}z_{k,k} &\textrm{for } 1\le k\leq m.
\end{aligned}\ee

(ii) If $n\le m$, then the equations above define $\check z_{k,\ell}$ for $1\le\ell\le k\wedge (n-1)$.  Set
    \be  \check z_{k,n}=a_{n,n}\dotsm a_{k,n} 
  \quad\text{ for $k=n,\dotsc,m$,} \label{NYalg2}\ee
    while 
$\check z_{k,\ell}$ for $\ell\ge n+1$ remain undefined.  

\medskip

\begin{prop}  Let $(w_{ij})$ be an $n\times m$ weight matrix and 
  $T=T(W)$   defined by \eqref{defT}.  Then the output 
$T$ is equivalent to the patterns $(P,Q)$  obtained from $n$ steps of 
 the Noumi-Yamada
  evolution,  through these equations: 
 \begin{align}
\label{tz1}  \text{$P$ pattern:} \ \  z_{k\ell}&= t_{n-\ell+1, \,k-\ell+1} , \quad
1\le\ell\le k\wedge n, \, 1\le k\le m \\[2pt]
\label{tz2} \text{$Q$ pattern:} \ \  z'_{s\ell}&= t_{s-\ell+1,\, m-\ell+1} , \quad
1\le\ell\le m\wedge s, \, 1\le s\le n. 
\end{align}  
\label{tzprop} \end{prop} 

\smallskip

Note in particular the common shape vector 
$z_{m\,\centerdot}=z'_{n\,\centerdot}=(t_{n-\ell+1,m-\ell+1})_{1\le \ell\le p}$.
Here is an illustration for $n\times m=3\times 6$. 
\be
T=\begin{bmatrix}   
z_{33} &z_{43} &z_{53} &z_{63}=z'_{33} &z'_{22} &z'_{11} \\
z_{22} &z_{32} &z_{42} &z_{52} &z_{62}=z'_{32} &z'_{21} \\
z_{11} &z_{21} &z_{31} &z_{41} &z_{51} &z_{61}=z'_{31} 
\end{bmatrix}   \\[8pt]
\label{T(W)Z}\ee

\begin{proof}[Proof of Proposition \ref{tzprop}]  We keep $m$ fixed and do induction on $n$.  In the case $n=1$,
the  $m$-vector $\check z_{\centerdot\,1}$ described by equation \eqref{NYalg2} is
the same
as that obtained by applying $R_1=\pi^m_1=l_{1m}\circ\dotsm\circ l_{11} $
to the top row $w_{1\,\centerdot}$ of the weight matrix.  

\smallskip

Suppose the statement is true for $T^{n-1,m}$. Add the $n$th weight row  $w_{n\,\centerdot}$
to $T^{n-1,m}$ and call the resulting  $n\times m$ matrix 
$\wt T^{n,m}=\big(\begin{smallmatrix} T^{n-1,m} \\ w_{n\,\centerdot} \end{smallmatrix}\bigr)$.  
   Then $T^{n,m}=R_n(\wt T^{n,m})$. 
From the definition of $R_n$ we see that on row $i\in\{1,\dotsc, n-1\}$ it alters
only elements $\tilde t_{ij}$ for $j-i\le m-n$.  Consequently after the application
of $R_n$, the induction assumption implies 
that \eqref{tz2} remains in force for $1\le s\le n-1$.  It only remains to 
check that \eqref{tz1} holds after the application
of $R_n$.  

Again we do induction, starting from the bottom row of $T^{n,m}$ and moving
up row by row. This corresponds to executing 
$R_n= \pi^{(m-n)\vee 0+1}_{(n-m)\vee 0+1}\circ\dotsm \circ \pi^{m-1}_{n-1}\circ \pi^m_n$
step by step.  

Before applying $\pi^m_n$, the two bottom rows of $\wt T^{n,m}$ are
\[
\wt T^{n,m} =\begin{bmatrix}  
&\dotsm  &\dotsm  &  \\
z_{11} &z_{21} &\dotsm  &z_{m1} \\
w_{n1} &w_{n2}  &\dotsm   &w_{nm}  
\end{bmatrix} 
=\begin{bmatrix}  
&\dotsm  &\dotsm  &  \\
z_{11} &z_{21} &\dotsm  &z_{m1} \\
a_{11} &a_{21}  &\dotsm   &a_{m1}  
\end{bmatrix} 
\]
where we used the first row of \eqref{NYalg1}. 
Apply $\pi^m_n=l_{nm}\circ l_{n,m-1}\circ\dotsm\circ l_{n1}$. Only   the 
bottom two rows are impacted.  Use the notation from \eqref{NYalg1}.  
\begin{align*}
&\begin{bmatrix}   z_{11} &z_{21} &z_{31}  &\dotsm  &z_{m1} \\
a_{11} &a_{21} &a_{31} &\dotsm   &a_{m1}  \end{bmatrix} 
\quad \stackrel{l_{n1}}{\longmapsto} \quad
\begin{bmatrix}    z_{11} &z_{21} &z_{31} &\dotsm  &z_{m1} \\
 \check z_{11} &a_{21} &a_{31} &\dotsm   &a_{m1}   \end{bmatrix}   \\[4pt]
 \stackrel{l_{n2}}{\longmapsto} \quad
&\begin{bmatrix}    a_{22} &z_{21} &z_{31} &\dotsm  &z_{m1} \\
 \check z_{11} &\check z_{21} &a_{31} &\dotsm   &a_{m1}   \end{bmatrix} 
 \quad  \stackrel{l_{n3}}{\longmapsto} \quad
\begin{bmatrix}    a_{22} &a_{32} &z_{31} &\dotsm  &z_{m1} \\
 \check z_{11} &\check z_{21} &\check z_{31} &\dotsm   &a_{m1}   \end{bmatrix} \\[4pt]
 \stackrel{l_{n4}}{\longmapsto} \quad &\dotsm \quad  \stackrel{l_{nm}}{\longmapsto} 
 \quad
 \begin{bmatrix}    a_{22} &a_{32} &a_{42} &\dotsm &a_{m2}  &z_{m1} \\
 \check z_{11} &\check z_{21} &\check z_{31} &\dotsm   &\check z_{m-1,1}  &\check z_{m1}   \end{bmatrix} 
\end{align*} 
Now the bottom row of $T^{n,m}$ is in place.  
Note that the transformations above left in place $z_{m1}=z'_{n1}$ as
they should, for this entry is in accordance with \eqref{tz2}. 

Next, an application of 
$\pi^{m-1}_{n-1}=l_{n-1,m-1}\circ l_{n-1,m-2}\circ\dotsm\circ l_{n-1,1}$
transforms rows $n-2$ and $n-1$ in this manner: 
 \begin{align*}
 \begin{bmatrix}  
 z_{22} &z_{32} &z_{42}  &\dotsm &z_{m-1,2}  & z'_{n2} &z'_{n-1,1} \\
   a_{22} &a_{32} &a_{42} &\dotsm &a_{m-1,2}  &a_{m2}  &z'_{n1} \\
 \check z_{11} &\check z_{21} &\check z_{31} &\dotsm  &\check z_{m-2,1}    &\check z_{m-1,1}  &\check z_{m1}   
 \end{bmatrix}  \\[6pt]
 \stackrel{\pi^{m-1}_{n-1}}{\longmapsto} \quad
  \begin{bmatrix}  
 a_{33} &a_{43} &a_{53}  &\dotsm &a_{m-2,3} & z'_{n2} &z'_{n-1,1} \\
   \check z_{22} &\check z_{32} &\check z_{42} &\dotsm &\check z_{m-1,2} &\check z_{m2}  &z'_{n1} \\
 \check z_{11} &\check z_{21} &\check z_{31} &\dotsm  &\check z_{m-2,1}   &\check z_{m-1,1}  &\check z_{m1}   \\[2pt]
 \end{bmatrix} 
\end{align*} 
The  bottom two  rows of $T^{n,m}$ are in place.  These steps  continue 
 until we arrive at $T^{n,m}$.  
\end{proof} 

\section{Symmetric input matrix}\label{sym-sec}

As it is needed in the following, we will write $R_i^{n,m}$ and $T=T^{n,m}$ for the mappings defined 
in \eqref{defR}--\eqref{defT}, and note the following recursive structure.  Let $W=(w_{ij})\in (\R_{>0})^{n\times m}$
and write $W_{k,m}=(w_{ij},\ 1\le i\le k,\ 1\le j\le m)$.  Recall that
$$T^{n,m}=R^{n,m}_n\circ R^{n,m}_{n-1} \circ \cdots \circ R^{n,m}_1.$$ 
Now, for each $i\le n$, the mapping $R^{n,m}_i$ acts only on the first $i$ rows of $W$ and leaves the 
remaining rows of $W$ unchanged.  In fact, for each $i\le k\le n$, we have
$$R^{n,m}_i (W) = \begin{pmatrix} R^{k,m}_i(W_{k,m}) \\ W_{k,m}^c \end{pmatrix} ,$$
where $W_{k,m}^c=(w_{ij},\ k+1\le i\le n,\ 1\le j\le m)$.  This property is immediate from the definitions.
This gives the basic recursion
\begin{equation}\label{rec}
T^{n,m}(W)=R_n^{n,m} \begin{pmatrix} T^{n-1,m}(W_{n-1,m}) \\ w_{n1}\ \ldots\ w_{nm} \end{pmatrix} .
\end{equation}
Recall that 
\begin{equation}\label{tran}
T^{m,n}(W^t)=[T^{n,m}(W)]^t.
\end{equation}  
In particular, if $n=m$ and $W$ is symmetric, then $T^{n,n}(W)$ is also symmetric.
\begin{lem}\label{lem-sym}
Suppose that $n=m$ and $W$ is symmetric.  

{\rm (a)} The following recursion holds:
\begin{equation}\label{recs}
T^{n,n}(W) = R^{n,n}_n \begin{pmatrix} \left[ R^{n,n-1}_n \begin{pmatrix} T^{n-1,n-1}(W_{n-1,n-1}) \\ 
w_{1n}\ \ldots\ w_{n-1,n} \end{pmatrix} \right]^t \\ w_{1n}\ \ldots\ w_{nn} \end{pmatrix} .
\end{equation}
Moreover, if we denote by $(s_{ij})$ the elements of the $(n-1)\times n$ matrix
\be S=\left[ R^{n,n-1}_n \begin{pmatrix} T^{n-1,n-1}(W_{n-1,n-1})\\ w_{1n}\ \ldots\ w_{n-1,n} \end{pmatrix} \right]^t \label{defS} \ee
and by $(t_{ij})$ the elements of $T^{n,n}(W)$, then 
\begin{equation}
\begin{array}{ll}
 t_{ij}=s_{ij} & \mbox{ for }1\le i<j\le n  \\
 t_{11}=s_{12}/2s_{11} & \\
 t_{ii}=s_{i,i+1}s_{i-1,i}/s_{ii} & \mbox{ for } 2\le i\le n-1\\
 t_{nn}=2s_{n-1,n}w_{nn}. &  
 \end{array}
\label{symm4}\end{equation}

{\rm (b)} For $n\ge 1$ we have this identity:  
\be
 4^{\fl{n/2}}  \prod_{i=1}^n w_{ii} 
\; = \; \frac{\displaystyle\prod_{j=0}^{\fl{\frac{n-1}2}} t_{n-2j, \,n-2j}}
 {\displaystyle\prod_{j=0}^{\fl{\frac{n-2}2}} t_{n-1-2j,\, n-1-2j}}\;=\; \frac{\displaystyle\prod_{\text{$i$ odd}} z_{ni}}{\displaystyle\prod_{\text{$i$ even}} z_{ni}}. 
\label{symm5}\ee

\end{lem}
\begin{proof} Part (a). 
Using (\ref{rec}), (\ref{tran}) and the fact the $W$ is symmetric,
\begin{align*}
T^{n,n}(W) &= R^{n,n}_n \begin{pmatrix} T^{n-1,n}(W_{n-1,n}) \\ w_{n1}\ \ldots\ w_{nn} \end{pmatrix} \\
&= R^{n,n}_n \begin{pmatrix} \left[ T^{n,n-1}([W_{n-1,n}]^t)\right]^t \\ w_{n1}\ \ldots\ w_{nn} \end{pmatrix} \\
&= R^{n,n}_n \begin{pmatrix} \left[ T^{n,n-1}(W_{n,n-1})\right]^t \\ w_{n1}\ \ldots\ w_{nn} \end{pmatrix} \\
&= R^{n,n}_n \begin{pmatrix} \left[ R^{n,n-1}_n \begin{pmatrix} T^{n-1,n-1}(W_{n-1,n-1}) \\ 
w_{n1}\ \ldots\ w_{n,n-1} \end{pmatrix} \right]^t \\ w_{n1}\ \ldots\ w_{nn} \end{pmatrix} \\
&= R^{n,n}_n \begin{pmatrix} \left[ R^{n,n-1}_n \begin{pmatrix} T^{n-1,n-1}(W_{n-1,n-1}) \\ 
w_{1n}\ \ldots\ w_{n-1,n} \end{pmatrix} \right]^t \\ w_{1n}\ \ldots\ w_{nn} \end{pmatrix} .
\end{align*}
This proves the first claim.  So we have
$$T^{n,n}(W) = R^{n,n}_n \begin{pmatrix} S \\ w_{1n}\ \ldots\ w_{nn} \end{pmatrix} ,$$
where $S\in (\R_{>0})^{(n-1)\times n}$.
To prove the second claim, first note that the mapping $R^{n,n}_n$ leaves the elements of
its input matrix which are strictly above the diagonal unchanged.  Thus, $ t_{ij}=s_{ij}$ for $1\le i<j\le n$.
Using this, the symmetry of $T$, and recalling how the row insertion procedure works (see Section~\ref{nvo}), 
we see that
$$t_{nn}=w_{nn} (t_{n-1,n}+s_{n-1,n})=2s_{n-1,n}w_{nn},$$
\begin{align*}
t_{n-1,n-1} &= \frac{t_{n-1,n} s_{n-1,n}}{s_{n-1,n-1}(t_{n-1,n}+s_{n-1,n})} (t_{n-2,n-1}+s_{n-2,n-1})\\
&=s_{n-1,n} s_{n-2,n-1} / s_{n-1,n-1} ,
\end{align*}
and so on; for $2\le i\le n-1$ we have $t_{ii}=s_{i,i+1}s_{i-1,i}/s_{ii}$ and then finally,
$$t_{11}=\frac{t_{12}s_{12}}{s_{11}(t_{12}+s_{12})}=s_{12}/2s_{11} ,$$
as required.

Part (b).  The second equality in \eqref{symm5} is a consequence of 
\eqref{tz1}.   The first equality is proved by induction on $n$.   Cases $n=2$ and 
$n=3$ are checked by hand. 

 Suppose  \eqref{symm5} is true for $n-1$.
Observe  first  from the definition of the mappings that 
$R^{n,n-1}_n$ operating on  $\bigl(\begin{smallmatrix} T^{n-1,n-1} \\ w_{1n}\ \ldots\ w_{n-1,n} \end{smallmatrix} \bigr)$ does not alter the diagonal
$\{ t^{n-1}_{ii}\}_{1\le i\le n-1}$ of $T^{n-1,n-1}$. 
Consequently \eqref{defS} implies 
 that $s_{ii}=t^{n-1}_{ii}$ for $1\le i\le n-1$.  
 
Suppose $n$ is even. Then the middle fraction of \eqref{symm5} 
develops as follows,  through equations  \eqref{symm4}, $s_{ii}=t^{n-1}_{ii}$
and by induction:   
\begin{align*}
\frac{t_{nn}t_{n-2,n-2}\dotsm t_{22}}{t_{n-1,n-1}t_{n-3,n-3}\dotsm t_{11}}
 &=\frac{\displaystyle  2s_{n-1,n}w_{nn} \cdot \frac{s_{n-2,n-1} s_{n-3,n-2}}{s_{n-2,n-2}}
 \dotsm \frac{s_{23} s_{12}}{s_{22}}  }  
 {\displaystyle \frac{s_{n-1,n} s_{n-2,n-1}}{s_{n-1,n-1}}  \cdot 
 \frac{s_{n-3,n-2} s_{n-4,n-3}}{s_{n-3,n-3}} \dotsm  \frac{s_{12}}{2s_{11}}   } \\[6pt]
&= 4w_{nn} \cdot \frac{s_{n-1,n-1}s_{n-3,n-3}\dotsm s_{11}}{s_{n-2,n-2}s_{n-4,n-4}\dotsm s_{22}} \\
&= 4w_{nn} \cdot 4^{\frac{n}2-1} \prod_{i=1}^{n-1} w_{ii}
=  4^{\fl{n/2}}  \prod_{i=1}^n w_{ii} . 
\end{align*} 
The case of odd $n$ develops similarly except that now the product
in the numerator finishes with ${s_{12}}/{2s_{11}}$ and consequently 
the factors of 2 cancel  each other. 
\end{proof}
\begin{thm}\label{jsym} Suppose that $n=m$ and $W$ is symmetric.  Then $T=T(W)=(t_{ij})$ is also symmetric,
and the Jacobian of the map 
$$(\log w_{ij},\ 1\le i\le j\le n)\mapsto (\log t_{ij},\ 1\le i\le j\le n)$$
is $\pm 1$. \end{thm}
\begin{proof}  We prove this by induction on $n$.  When $n=2$, we have $t_{11}= w_{12}/2$, 
$t_{12}=w_{11} w_{12}$, $t_{22}=2w_{11} w_{12} w_{22}$ and the result is immediate.
Now, by the previous lemma,
$$T=R^{n,n}_n
\begin{pmatrix}
\left[ R^{n,n-1}_n \begin{pmatrix} T^{n-1,n-1}(W_{n-1,n-1})\\ w_{1n}\ \ldots\ w_{n-1,n} \end{pmatrix} \right]^t \\
w_{1n}\ \ldots\ w_{n-1,n}\ w_{nn}
\end{pmatrix} .$$
Denoting by $(s_{ij})$ the elements of the matrix
$$S=\left[ R^{n,n-1}_n \begin{pmatrix} T^{n-1,n-1}(W_{n-1,n-1})\\ w_{1n}\ \ldots\ w_{n-1,n} \end{pmatrix} \right]^t,$$
we have, by the previous lemma, 
\begin{equation}\label{rem}
\begin{array}{ll}
 t_{ij}=s_{ij} & \mbox{ for } 1\le i<j\le n  \\
 t_{11}=s_{12}/2s_{11} & \\
 t_{ii}=s_{i,i+1}s_{i-1,i}/s_{ii} & \mbox{ for } 2\le i\le n-1\\
 t_{nn}=2s_{n-1,n}w_{nn} &  
 \end{array}
\end{equation}
This expresses the $n(n+1)/2$ variables $t_{ij},\ 1\le i\le j\le n$ as a function, which we shall
denote by $F$, of the $n(n+1)/2$ variables $s_{ij}, 1\le i<j\le n$ and $s_{11},\ldots,s_{n-1,n-1},w_{nn}$.

Denote by $t^{n-1}_{ij}$ the elements of the symmetric matrix $T^{n-1,n-1}(W_{n-1,n-1})$.
By the induction hypothesis, the map $$(\log w_{ij},\ 1\le i\le j\le n-1)\mapsto (\log t^{n-1}_{ij},\ 1\le i\le j\le n-1)$$
has Jacobian $\pm 1$.  The mapping $R^{n,n-1}_n$ on the whole of $(\R_{>0})^{n\times (n-1)}$
is a composition of $l_{ij}$-maps and hence has Jacobian $\pm 1$ in logarithmic variables;
since it leaves matrix elements above the diagonal unchanged, its restriction to the space of matrix
elements on and below the diagonal also has Jacobian $\pm 1$ in logarithmic variables.
It follows that the mapping
\begin{align*}
(\log w_{ij}, 1\le i\le j <n;\ & \log w_{in}, 1\le i<n) \\
& \mapsto (\log s_{ij}, 1\le i<j\le n;\ \log s_{ii}, 1\le i <n)
\end{align*}
has Jacobian $\pm 1$.  It therefore remains only to show that the Jacobian sub matrix of the map 
$F$ (in logarithmic variables) which corresponds to the variables 
$(\log s_{11},\ldots,\log s_{n-1,n-1},\log w_{nn})$ and $(\log t_{11},\ldots,\log t_{nn})$ has determinant $\pm 1$.  
From (\ref{rem}), this sub matrix is given by
$$
\bordermatrix{ & \log s_{11} & \log s_{22} & \cdots & \log s_{n-1,n-1} & \log w_{nn} \cr
\log t_{11} & -1 &&&&\cr
\log t_{22} & &-1&&&\cr
\vdots &&&\ddots&&\cr
\log t_{n-1,n-1} & &&&-1&\cr
\log t_{n,n} & &&&&1\cr
},
$$
which completes the proof.
\end{proof}

\def\dipa{\zeta}

Consider the measure on symmetric input matrices with positive entries defined by
\begin{eqnarray}\label{weightsym}
\nu_{\alpha,\dipa}(dw) = \prod_{i<j} w_{ij}^{-\alpha_i-\alpha_j} \prod_i w_{ii}^{-\alpha_i-\dipa}
\exp\left(-\sum_{i<j} \frac1{w_{ij}} -  \sum \frac1{2w_{ii}} \right) \prod_{i\le j} \frac{dw_{ij}}{w_{ij}} ,
\end{eqnarray}
where $\alpha\in\R^n$ and $\dipa\in\R$ satisfy $\alpha_i+\dipa>0$ for each $i$
and $\alpha_i+\alpha_j>0$ for $i\ne j$.
Note that
$$\int_{(\R_{>0})^{n(n+1)/2} }\nu_{\alpha,\dipa}(dw)
 = 2^{\sum_{i=1}^n(\alpha_i+\dipa)} \prod_i \Gamma(\alpha_i+\dipa) \prod_{i<j}\Gamma(\alpha_i+\alpha_j) .$$
In this setting we have $R=C$ and so, using (\ref{pq}) and Lemma \ref{lem-sym}(b),
$$\prod_{i<j} w_{ij}^{-\alpha_i-\alpha_j} \prod_i w_{ii}^{-\alpha_i-\dipa} 
= 4^{\lfloor n/2\rfloor \dipa}\prod_i z_{ni}^{(-1)^i\dipa}R^{-\alpha}.$$
Thus, by Theorems \ref{bi} and \ref{jsym}, we obtain the following result.
\begin{cor}\label{pfs} The push-forward of $\nu_{\alpha,\dipa}$ under $\sigma$ is given by
\be\nu_{\alpha,\dipa}\circ\sigma^{-1} (dx) = 
4^{\lfloor n/2\rfloor \dipa} f(x)^\dipa
e^{-\frac{1}{2x_n}} \Psi^n_\alpha(x) \prod_{i=1}^n \frac{dx_i}{x_i} , \label{nu4}\ee
where
$$f(x)=\prod_i x_i^{(-1)^i}.$$
\end{cor}


If $\lambda\in\C^n$ and $\gamma\in\C$
satisfy $\Re(\lambda_i+\gamma)>0$ for each $i$,
and $\Re(\lambda_i+\lambda_j)>0$ for $i\ne j$, then
\begin{align*}
\int_{(\R_{>0})^n}  f(x)^\gamma & e^{-\frac{1}{2x_n}} \Psi^n_\lambda(x) \prod_{i=1}^n \frac{dx_i}{x_i} \\
& = 4^{-\lfloor n/2\rfloor \gamma}
2^{\sum_{i=1}^n(\lambda_i+\gamma)} \prod_i\Gamma(\lambda_i+\gamma) \prod_{i<j}\Gamma(\lambda_i+\lambda_j) .
\end{align*}
Now, using (\ref{a}) we can strengthen this to:
\begin{cor} Suppose $\lambda\in\C^n$ and $\gamma\in\C$
satisfy $\Re(\lambda_i+\gamma)>0$ for each $i$,
and $\Re(\lambda_i+\lambda_j)>0$ for $i\ne j$. 
Then, for $s>0$,
\begin{align*}
\int_{(\R_{>0})^n} & f(x)^\gamma  e^{-s/x_n} \Psi^n_\lambda(x) \prod_{i=1}^n \frac{dx_i}{x_i} \\
&= c_n(s,\gamma) s^{-\sum_{i=1}^n\lambda_i} 
\prod_i \Gamma(\lambda_i+\gamma) \prod_{i<j}\Gamma(\lambda_i+\lambda_j) ,
\end{align*}
where 
$$c_n(s,\gamma)=\begin{cases} 1 & \mbox{ if $n$ is even,}\\
s^{-\gamma} & \mbox{ if $n$ is odd.}
\end{cases}$$
\label{cor-sym-3}\end{cor}
By (\ref{inv}) this is equivalent to the following identity which is equivalent to an integral identity
conjectured by Bump and Friedberg~\cite{bf} and proved by Stade~\cite[Theorem 3.3]{stade-jams},
see Theorem~\ref{st3} below.  We note that in \cite{stade-jams} the corresponding statement
is proved without any restrictions on the parameters.  This integral is associated
with an archimedean $L$-factor of an exterior square automorphic $L$-function on $GL(n,\R)$.
\begin{cor}[Stade]\label{nwid} Suppose $\lambda\in\C^n$ and $\gamma\in\C$
satisfy $\Re(\lambda_i+\gamma)>0$ for each $i$,
and $\Re(\lambda_i+\lambda_j)>0$ for $i\ne j$. 
Then, for $s>0$,
\begin{align*}
\int_{(\R_{>0})^n} & f(x')^\gamma e^{-sx_1} \Psi^n_{-\lambda}(x) \prod_{i=1}^n \frac{dx_i}{x_i} \\
& = c_n(s,\gamma) s^{-\sum_{i=1}^n\lambda_i} 
\prod_i \Gamma(\lambda_i+\gamma) \prod_{i<j}\Gamma(\lambda_i+\lambda_j) ,
\end{align*}
where $x'_i=1/x_{n-i+1}$.
\end{cor}
Note that $f(x')=f(x)$ if $n$ is even and $f(x')=1/f(x)$ if $n$ is odd.

Now, consider the probability measure on symmetric matrices with positive entries defined by
 \be\tilde\nu_{\alpha,\dipa}(dw)=Z_{\alpha,\dipa}^{-1} \nu_{\alpha,\dipa}(dw),
 \label{nutilde}\ee
 where
$$Z_{\alpha,\dipa}=2^{\sum_{i=1}^n(\alpha_i+\dipa)} \prod_i \Gamma(\alpha_i+\dipa) \prod_{i<j}\Gamma(\alpha_i+\alpha_j).$$
From Corollary \ref{pfs}, we obtain: 
\begin{cor}\label{parsym}
The Laplace transform of the law of the polymer partition function $t_{nn}$ under $\tilde\nu_{\alpha,\dipa}$ 
is given for $r>0$ by
 $$\int e^{-r t_{nn}} \tilde\nu_{\alpha,\dipa}(dw)= 4^{\lfloor n/2\rfloor \dipa} Z_{\alpha,\dipa}^{-1} 
 \int_{(\R_{>0})^n} f(x)^\dipa
 e^{-rx_1-\frac{1}{2x_n}} \Psi^n_\alpha(x) \prod_{i=1}^n \frac{dx_i}{x_i}.$$
\end{cor} 


{\bf Remark: A formal computation.} We formally rewrite the above integral formula in terms of a multiple contour
integral that should be amenable to asymptotic analysis.
Let $\epsilon>0$ and set $\alpha_i'=\alpha_i+\epsilon$.
It follows from Corollary \ref{bs1} (or~\ref{L2})
that the function $ e^{-\frac{1}{2x_n}} \Psi^n_{\alpha'}(x)$ is in $L_2((\R_{>0})^n,\prod_{i=1}^n dx_i/x_i)$.
Moreover, by Corollary \ref{bs1}, for $\lambda\in\iota\R^n$,
\begin{equation}\label{t1}
\int_{(\R_{>0})^n} e^{-\frac{1}{2x_n}} \Psi^n_{\alpha'}(x) \Psi^n_\lambda(x) \prod_{i=1}^n \frac{dx_i}{x_i}
= 2^{\sum_i(\lambda_i+\alpha_i+\epsilon)} \prod_{i,j} \Gamma(\alpha_i+\lambda_j+\epsilon).
\end{equation}
Thus, by the Plancherel theorem, for any $g\in L_2((\R_{>0})^n,\prod_{i=1}^n dx_i/x_i)$ we can write
\begin{align}\label{pl}
 \int_{(\R_{>0})^n}  \overline{g(x)} e^{-\frac{1}{2x_n}} & \Psi^n_{\alpha'}(x)  \prod_{i=1}^n \frac{dx_i}{x_i} \nonumber \\
& =  \int_{\iota \R^n} \overline{ \hat g(\lambda) }
2^{\sum_i(\lambda_i+\alpha_i+\epsilon)} \prod_{i,j} \Gamma(\alpha_i+\lambda_j+\epsilon) s_n(\lambda) d\lambda .
\end{align}

Suppose $n$ is even.  By Corollary \ref{nwid}, if $r>0$ and $\Re\lambda_i>0$ for each $i$, 
\begin{align*} 
\int_{(\R_{>0})^n}  f(x)^\dipa  e^{-rx_1} & \Psi^n_{-\lambda}(x)  \prod_{i=1}^n \frac{dx_i}{x_i} \\
&= r^{-\sum_{i=1}^n\lambda_i} \prod_i \Gamma(\lambda_i+\dipa) \prod_{i<j}\Gamma(\lambda_i+\lambda_j).
\end{align*}
By (\ref{shift}) it follows that, for $\epsilon>0$ and $\lambda\in\iota\R^n$,
\begin{align} \label{t2}
\int_{(\R_{>0})^n}  f(x)^\dipa  e^{-rx_1} & \left(\prod_i x_i^\epsilon\right) \Psi^n_{-\lambda}(x)  \prod_{i=1}^n \frac{dx_i}{x_i} \\
&= r^{-\sum_{i=1}^n(\lambda_i+\epsilon)} \prod_i \Gamma(\lambda_i+\dipa+\epsilon) \prod_{i<j}
\Gamma(\lambda_i+\lambda_j+2\epsilon).\nonumber
\end{align}

Formally, combining (\ref{t1}), (\ref{t2}) and (\ref{pl}) yields the following integral formula for the 
Laplace transform of the law of the polymer partition function $t_{nn}$ 
under the probability measure $\tilde\nu_{\alpha,\dipa}$ :
\begin{align}\label{ifs-even}
 \int    e^{-r t_{nn}}   \tilde\nu_{\alpha,\dipa}(dw) 
 = \int_{\iota\R^n} \left(\frac{r}2\right)^{-\sum_i(\lambda_i+\epsilon)} 
\prod_i & \frac{\Gamma(\lambda_i+\dipa+\epsilon)}{\Gamma(\alpha_i+\dipa)}
 \prod_{i,j} \Gamma(\alpha_i+\lambda_j+\epsilon) \\
& \times \prod_{i<j} \frac{\Gamma(\lambda_i+\lambda_j+2\epsilon)}
{\Gamma(\alpha_i+\alpha_j)}
s_n(\lambda) d\lambda  \nonumber
\end{align} 
or, equivalently,
\begin{align}\label{ifs}
 \int   & e^{-r t_{nn}}   \tilde\nu_{\alpha,\dipa}(dw) \\
& = \int \left(\frac{r}2\right)^{-\sum_i\lambda_i} \prod_i \frac{\Gamma(\lambda_i+\dipa)}{\Gamma(\alpha_i+\dipa)}
\prod_{i,j} \Gamma(\alpha_i+\lambda_j) \prod_{i<j} \frac{\Gamma(\lambda_i+\lambda_j)}{\Gamma(\alpha_i+\alpha_j)}
s_n(\lambda) d\lambda  \nonumber
\end{align} 
where the integration is along vertical lines with $\Re\lambda_i>0$ for each $i$.
If $n$ is odd, we similarly formally obtain, this time using Theorem~\ref{thm-stade2-psi} instead of Corollary~\ref{nwid}
because in this case $f(x')^\dipa=f(x)^{-\dipa}$ and $\dipa>0$, 
\begin{align}\label{ifs-odd}
 \int   & e^{-r t_{nn}}   \tilde\nu_{\alpha,\dipa}(dw) \\
& = \int \left(\frac{r}2\right)^{-\sum_i\lambda_i} \prod_i \frac{\Gamma(\lambda_i-\dipa)}{\Gamma(\alpha_i+\dipa)}
\prod_{i,j} \Gamma(\alpha_i+\lambda_j) \prod_{i<j} \frac{\Gamma(\lambda_i+\lambda_j)}{\Gamma(\alpha_i+\alpha_j)}
s_n(\lambda) d\lambda  \nonumber
\end{align} 
where the integration is along vertical lines with $\Re\lambda_i>0$ for each $i$.
It seems reasonable to expect the integral formulas (\ref{ifs}) and (\ref{ifs-odd})
 to be valid, at least in some suitably regularised sense.

\section{Geometric RSK for triangular arrays and paths below a hard wall}

\label{sec:wall}  

In this section we introduce a   birational, geometric RSK type mapping $T^\tr_n$ that maps triangular arrays $X_n=(x_{ij}, \,  1\leq j < i\leq n)$ to triangular arrays $T=(t_{ij}, \,   1\leq j< i\leq n)$, both with positive real entries. The motivation comes from the  symmetric polymer of Section \ref{sym-sec}, with a (de)pinning parameter $\zeta$ that tends to infinity. This will become clear later on in Proposition \ref{lot} (see also the remarks at the end of the section). Notions like the type and the shape can  be defined also for this mapping.  We prove that it satisfies a version of the fundamental identity \eqref{bi} and   preserves volume in logarithmic variables. Moreover, we can relate the shape to partition
functions of nonintersecting paths below a ``hard wall", that is, paths restricted to $\{(i,j)\colon j<i\}$.

For $n=2$ the mapping is defined by  
\begin{eqnarray}\label{deftrn2}
T^\tr_2\left( x_{21}\right)=x_{21}. 
\end{eqnarray}
For $n\geq 3$ we define inductively
\begin{eqnarray}\label{defttr}
T^\tr_n(X_n)=R^\tr_n 
 \begin{pmatrix} T^\tr_{n-1}(X_{n-1}) 
 \\ x_{n1}\ \ldots\ x_{n,n-1} \end{pmatrix},
\end{eqnarray}
with $X_{n-1}=(x_{ij}, \,\,1\leq j < i\leq n-1)$ and
\begin{eqnarray}\label{defrtr}
 R^\tr_n= \rho^{\tr,n}_{n-1}\circ \cdots \circ \rho^{\tr,n}_1
\end{eqnarray}
where  
\be \label{bktr}\begin{aligned} 
\rho^{\tr,n}_j&=\rho^{n}_j\qquad
\text{ for $j=1,\dotsc,n-2$, }  \\[1pt]
\text{ and } \quad 
\rho^{\tr,n}_{n-1}&=b^{\tr, n}_{2,1}\circ \cdots \circ b^{\tr, n}_{n-1,n-2}\circ b^{\tr, n}_{n,n-1}\circ r^\tr_{n,n-1},
\end{aligned}\ee
and $\rho^{n}_j$ is defined in \eqref{rodef}.
To complete the definition of   $T^\tr_n$ we   define the mappings $b^{\tr, n}_{j,j-1}$  and   $r^\tr_{n,n-1}$  on  a triangular array $X_n=(x_{ij}, \,  1\leq j< i\leq n)$.
This is done as follows.  The mapping $r^\tr_{n,n-1}$ replaces $x_{n,n-1}$ by $1/x_{n,n-1}$. Observing the conventions   $x_{i0}=x_{n+1,n-1}=1$, make these definitions: 
\begin{itemize}
\item[$\bullet$] For $k=0,1,2,\dotsc,\fl{\frac{n}2}-1$,  $b^{\tr, n}_{n-2k,n-2k-1}$  replaces $x_{n-2k,n-2ki-1}$ with 
\begin{eqnarray}\label{bii-1}
x'_{n-2k,n-2k-1}&=&\frac{x_{n-2k+1,n-2k-1} \,\,x_{n-2k,n-2k-2}}{x_{n-2k,n-2k-1}}. 
\end{eqnarray}
 \item[$\bullet$]
For $k=1,2,\dotsc,\fl{\frac{n-1}2}$,  $b^{\tr, n}_{n-2k+1,n-2k}$  is the identity mapping.
\end{itemize}
 We  present explicitly  the cases $n=3,4$ to clarify the definitions. For $n=3$, 
\begin{align*}
T^\tr_3\left( \begin{array}{ccc} 
x_{21}\\
x_{31}& x_{32}\\
\end{array}\right) &=
\rho^{\tr,3}_2\circ \rho^{\tr,3}_1  
 \begin{pmatrix} T^\tr_{2}(x_{21})
 \\ x_{31}\,\,\,\,\,\, x_{32}\end{pmatrix} 
 =
\rho^{\tr,3}_2\circ \rho^{\tr,3}_1  
\left( \begin{array}{cc}
x_{21} & \\ x_{31} & x_{32}
\end{array} \right)\\
&=
\rho^{\tr,3}_2  
\left( \begin{array}{cc}
x_{21} & \\ x_{21}x_{31} & x_{32}
\end{array} \right)  
=
 \left( \begin{array}{cc}
x_{21} & \\ x_{21}x_{31} & x_{21}x_{31}x_{32}
\end{array} \right).
\end{align*}
For $n=4$, 
\begin{align*}
&T^\tr_4\left( \begin{array}{ccc} 
x_{21}\\
x_{31}& x_{32}\\
x_{41}& x_{42} & x_{43}\\
\end{array}\right) 
= 
\rho^{\tr,4}_3\circ \rho^{\tr,4}_2 \circ \rho^{\tr,4}_1 
 \begin{pmatrix} T^\tr_3
 \left( \begin{array}{cc} 
x_{21}& \\
x_{31}& x_{32}\\
\end{array}
 \right) 
 \\ x_{41}\,\,\,\,\,\, x_{42}\,\,\,\,\,\, x_{43} \end{pmatrix}\\
 &\qquad\qquad =
 \rho^{\tr,4}_3\circ \rho^{\tr,4}_2 \circ \rho^{\tr,4}_1
 \left( \begin{array}{ccc} 
x_{21}\\
x_{21}x_{31}& x_{21}x_{31}x_{32}\\
x_{41}& x_{42} & x_{43}\\
\end{array}\right)\\
&\qquad\qquad =
 \rho^{\tr,4}_3\circ \rho^{\tr,4}_2 
  \left( \begin{array}{ccc} 
x_{21}\\
x_{21}x_{31}& x_{21}x_{31}x_{32}\\
x_{21}x_{31}x_{41}& x_{42} & x_{43}\\
\end{array}\right)\\
&\qquad\qquad =
 \rho^{\tr,4}_3
  \left( \begin{array}{ccc} 
x_{21}\\
\frac{x_{21}x_{32}x_{41}}{x_{32}+x_{41}}& x_{21}x_{31}x_{32}\\
x_{21}x_{31}x_{41}& x_{21}x_{31}x_{42}(x_{32}+x_{41}) & x_{43}\\
\end{array}\right)\\
&\qquad\qquad =
   \left( \begin{array}{ccc} 
\frac{x_{32}x_{41}}{x_{32}+x_{41}}\\[2pt] 
\frac{x_{21}x_{32}x_{41}}{x_{32}+x_{41}}& x_{21}x_{31}x_{32}\\
x_{21}x_{31}x_{41}& x_{21}x_{31}x_{42}(x_{32}+x_{41}) & x_{21}x_{31}x_{42}x_{43}(x_{32}+x_{41}) \\
\end{array}\right).   
\end{align*}

\medskip

\smallskip 
 
For a triangular array $X=(x_{ij}, \, 1\leq j< i \leq n)$  define  
$$\E^\tr(X)=\frac{1}{x_{21}}+\sum_{j\leq i-1} \frac{x_{i-1,j}+x_{i,j-1}}{x_{ij}},
$$
with the convention that $x_{i0}=x_{ii}=0$ for $i=1,\dotsc,n$.
Here is  the analogue of Theorem \ref{bi} for   triangular arrays. 
\begin{theorem}\label{thmtr2}
Let $W_n=(w_{ij},\,\, 1\leq j<i \leq n\,)$ with $ w_{ij}\in \mathbb{R}_{>0}$.   Then the output array $T_n=T^\tr_n(W)$ satisfies
\begin{eqnarray}\label{thm022}
\E^\tr(T_n)=\sum_{1\leq j<i\leq n}\frac{1}{w_{ij}}.
\end{eqnarray}
\end{theorem}  

\begin{proof}

 We will show that 
$$\E^{\tr}(T_n)=\E^{\tr}(T^\tr_{n-1}(W_{n-1}))+\sum_{j=1}^{n-1}\frac{1}{w_{nj}}.
$$
To this end, let  $T^{0}=T^\tr_{n-1}(W_{n-1})$  and  $T^{k}=\rho^{\tr,n}_k\circ\cdots \circ\rho^{\tr,n}_1(T^\tr_{n-1}(W_{n-1}))$ for $k=1,2,\dotsc,n-1$.  For a triangular array $X$  define
$$\E^{\tr,n,k}(X)=\frac{1}{x_{21}}+\sum_{i,j}^{(k)}\frac{x_{i-1,j}+x_{i,j-1}}{x_{ij}}+\sum_{j=k+1}^{n-1}\frac{1}{x_{ij}},
$$
where summation $\sum_{ij}^{(k)}$ is over all indices $(i,j)$ such that $1\leq j< i \leq n$, but
 $(i,j)\neq (n,k+1),\dotsc,(n,n-1)$.  The boundary conventions $x_{i0}=x_{ii}=0$   are still in force.
  We will show that 
 $$\E^{\tr,n,k}(T^k)=\E^{\tr,n,k-1}(T^{k-1})\qquad \text{for} \,\, k=1,2,\dotsc,n-1, 
 $$
 and this will conclude the proof. Notice that for $k=1,2,\dotsc,n-2$   this fact   is already included in the proof of Theorem \ref{bi}, since $\rho^{\tr,n}_i=\rho^n_i$ for $i\leq n-2$. To check the case $k=n-1$, let   $X=T^{n-2} $
 and   $X'=\rho^{\tr,n}_{n-1}(X)=T^{n-1}$.  Since  $\rho^{\tr,n}_{n-1}$ alters only the elements $x_{i,i-1}, i=2,\dotsc,n$, and leaves the rest unchanged, 
 \begin{align}
 \E^{\tr,n,n-1}(X')&=\E^{\tr,n,n-1}(\rho^{\tr,n}_{n-1}(X))
 = \frac{1}{x'_{2,1}}+\sum_{j<i}\frac{x'_{i-1,j}+x'_{i,j-1}}{x'_{ij}}\nn\\
 &= \tilde\sum_{ j<i-1}\frac{x_{i-1,j}+x_{i,j-1}}{x_{ij}}\nn\\
 &\quad + \; \frac{1}{x'_{2,1}} +  \frac{x'_{2,1}}{x_{3,1}}+
  \sum_{i=3}^{n-1} \biggl(   \frac{x_{i,i-2}}{x'_{i,i-1}} + \frac{x'_{i,i-1}}{x_{i+1,i-1}}  \biggr)  +\frac{x_{n,n-2}}{x'_{n,n-1}}
  \label{row11}  
 \end{align}
 where in the summation $\tilde\sum$ we set appearances of terms $x_{i,i-1}, i=2,\dotsc,n$, equal to zero. 
 Consider the three parts of line \eqref{row11}.   

First  
 $$\frac{1}{x'_{21}}+\frac{x'_{21}}{x_{31}}=\frac{1}{x_{21}}+\frac{x_{21}}{x_{31}}
 $$
 because either  
  $n$ is odd and  $x'_{21}=x_{21}$, or  $n$ is even and $x'_{21}=x_{31}/x_{21}$. 
The middle terms satisfy 
\[   \frac{x_{i,i-2}}{x'_{i,i-1}} + \frac{x'_{i,i-1}}{x_{i+1,i-1}} 
  =     \frac{x_{i,i-2}}{x_{i,i-1}} + \frac{x_{i,i-1}}{x_{i+1,i-1}}\,, 
  \]   
either   by virtue of \eqref{bii-1} if  $i=n-2k$, or because $x'_{i,i-1}=x_{i,i-1}$   
when $i=n-2k+1$.    Finally,  
  \begin{eqnarray*}
 \frac{x'_{n,n-2}}{x'_{n,n-1}}=\frac{1}{x_{n,n-1}}
 \end{eqnarray*}
 by \eqref{bii-1} and the definition of $r^\tr_{n,n-1}$.   
 Making these substitutions on  line \eqref{row11}  converts  $ \E^{\tr,n,n-1}(X')$
 into $\E^{\tr,n,n-2}(X)$ and completes the proof. 
\end{proof}

The following theorem states the volume preserving property of the map $T^\tr_n$.
It follows from the volume preservation  of the individual steps in \eqref{bktr}.
\begin{theorem}\label{thmtr1}
Let   $W=(w_{ij},\,\, 1\leq j < i \leq n)\in(R_{>0})^{n(n-1)/2}$ as above,  and 
consider 
  the mapping $W\mapsto T^\tr_n(W)= (t_{ij},\,1\leq j < i\leq n)$.
 In logarithmic variables
$$ (\log w_{ij},\,1\leq j < i\leq n) \mapsto  (\log t_{ij},\,1\leq j < i\leq n)
$$ 
has Jacobian equal to  $\pm 1$.
\end{theorem} 

Consider a triangular array $W=(w_{ij},\,\,1\leq j<i\leq n)$ and the 
output pattern $P^\tr=T^\tr_n(W)=(t_{ij},\, 1\leq j<i \leq n)$.  	
The {\sl shape} of the pattern $P^\tr$ is defined as
\begin{eqnarray*}
sh\, P^\tr &=&sh\, T^\tr_n(W)= (t_{n,n-1},t_{n-1,n-2},\dotsc,t_{21})\\
\end{eqnarray*}

\begin{figure}[t]
 \begin{center}
\begin{tikzpicture}[scale=.7,rotate=90]
\draw[help lines] (1,1) grid (8,8);
\draw [dashed] (1,1) -- (8,8);
\draw[very thick] (1,2)--(1,3)--(1,4)--(1,5)--(2,5)--(2,6)--(2,7)--(3,7)--(3,8)--(4,8)--(5,8)--(6,8)--(7,8);
\draw[very thick] (2,3)--(2,4)--(3,4)--(3,5)--(4,5)--(4,6)--(4,7)--(5,7)--(6,7);
\draw [fill] (1,2) circle [radius=0.1];
\draw [fill] (2,3) circle [radius=0.1];
\draw [fill] (7,8) circle [radius=0.1];
\draw [fill] (6,7) circle [radius=0.1];
\draw [white, fill=white] (0.9,0.9)  -- (8.1,8.1) -- (8.1,.9) -- (.9,.9);

\end{tikzpicture}

\end{center}  
\caption{\small A pair $(\pi_1,\pi_2)\in\Pi^{(2)}_8$. We have used matrix representation, as opposed to Cartesian coordinates. On the upper left the paths begin at  $(2,1)$ and
$(3,2)$.  On the lower right the paths end at  $(8,7)$ and $(7,6)$.  The diagonal (dashed line) runs from $(1,1)$ to $(8,8)$. } \label{d:fig}
\end{figure}
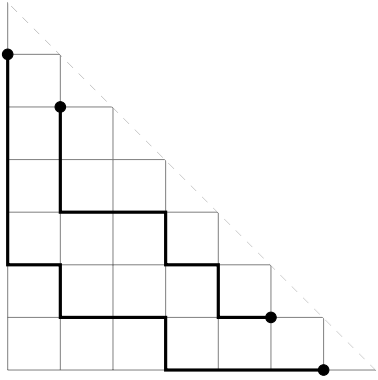

Our next goal is to relate the shape   to ratios of partition functions. 
  Let $\Pi^{(r)}_n$ be the collection of $r$-tuples of non-intesecting
nearest-neighbor lattice paths $\pi_1,\dotsc,\pi_r$ that start 
at positions $(2,1), (3,2),\dotsc, (r+1,r)$, end at positions
$(n,n-1), (n-1,n-2), \dotsc, (n-r+1, n-r)$, and stay strictly below the
diagonal in the matrix picture, i.e. never leave the set $\{(i,j): 1\le j<i\le n\}$.  See Figure \ref{d:fig}.  
Naturally $1\le r\le n/2$.  
Denote the partition sums by
\begin{eqnarray}\label{parttr}
z_r = \sum_{(\pi_1,\ldots,\pi_r)\in\Pi^{(r)}_{n}} \prod_{(i,j)\in \pi_1\cup\cdots\cup\pi_r} w_{ij}. 
\label{d:Z1}
\end{eqnarray}
The definition includes the case of a path
consisting of a single point, which happens when $n$ is even and $r=n/2$.
  
The next theorem states that the odd coordinates of the shape vector $sh \,P^\tr$ are given by  ratios of partition functions. 
\begin{theorem}\label{shape-of-tr}
Consider a triangular array $W=(w_{ij},\,\,1\leq j<i\leq n)\in(\R_{>0})^{n(n-1)/2}$, the output pattern $P^\tr=T^\tr_n(W)=(t_{ij},\,\,1\leq j<i\leq n)$ and the partition functions $z_r,\,\,r=1,2,\dotsc,\fl{n/2}$ as defined in \eqref{parttr}. Then 
\begin{eqnarray*}
(t_{n,n-1}, t_{n-2,n-3},\dotsc,t_{n-2\fl{n/2}+2,n-2\fl{n/2}+1})
=(z_1, z_2/z_1, \dotsc , z_{\fl{n/2}}/z_{\fl{n/2}-1}). 
\end{eqnarray*}
\end{theorem}
 The proof of this theorem will be presented after Proposition \ref{lot} below. 
Define an operator  $
\Lambda^\varen_n$ acting on $n\times n $ matrices $W$ by  
\begin{eqnarray*}
w_{ij} &\mapsto& \varen w_{ij}, \qquad\qquad\qquad\qquad\,\,\, i\neq j\\
w_{n-2i,n-2i}&\mapsto&2 \varen^2 w_{w_{n-2i,n-2i}}, \,\,\qquad\qquad i=0,1,\dotsc,\fl{n/2}-1\\
w_{n-2i+1,n-2i+1} &\mapsto&\frac{1}{2} w_{n-2i+1,n-2i+1}, \qquad\,\,\,\,\,\,\,\, i=1,2,\dotsc,\fl{n/2}-1,
\end{eqnarray*}
and $w_{11}\mapsto \frac{1}{2}w_{11}$, if $n$ is even, while $w_{11}\mapsto \varen w_{11}$, if $n$ is odd.
\vskip 4mm
Let $W_n^\tr=(w_{ij}^{\tr,n},\,\, 1\leq j< i\leq n)$ be a given triangular array.   
Let $W_n^\varepsilon$ be the symmetric  $n\times n$  matrix with $w_{ii}^\varen=\varen$ for $1\le i\le n$ and  $w_{ij}^\varen=w_{ij}^{\tr,n}$ for $1\leq j< i\leq n$.  
   Finally,   denote by $T^\did_n(W_n^\tr)=(w^{\did,n}_{ij},\,\,\, 1\leq i,j\leq n)$ a  symmetric $n\times n$  output matrix whose  lower triangular part $(t_{ij},\,\, 1\leq j< i\leq n)$ agrees with the output array $T^{\tr}_n(W^\tr_n)$, while the diagonal elements $(t_{ii})_{i=1,\dotsc,n}$ are determined by 
\be \label{collapse}\begin{aligned}
&t_{n-2k,n-2k}=t_{n-2k-1,n-2k-1}=t_{n-2k,n-2k-1} \quad \text{for}  \ \ k=0,1,\dotsc\\
\text{and} \quad  &t_{11}=1 \quad\text{if $n$ is odd.}  
\end{aligned}\ee

\begin{proposition} \label{lot}
Let $T^{n,n}$ be the geometric RSK mapping on $n\times n$ matrices with positive entries, defined in \eqref{defT}, and $W^\varen_n,\,\Lambda^\varen_n,\, T^\did_n,\,W^\tr_n$  as above. Then, as $\varen\searrow 0$,   
\begin{eqnarray}\label{loteq}
T^{n,n}(W^\varen_n)=\Lambda^\varen_n\circ T^\did_n(W^\tr_n)+  S_n^\varen
\end{eqnarray}
where $S_n^\varen$ is an   $n\times n$ matrix of  lower order terms,
specifically 
\be\label{l.o.t.}\begin{aligned}
(S_n^\varen)_{ij}&=o(\varen) ,  \quad \; \, i\ne j\\ 
  (S_n^\varen)_{n-2i,n-2i}&= o(\varen^2), \quad i=0,1,\dotsc,\fl{n/2}-1\\
  (S_n^\varen)_{n-2i+1,n-2i+1}&= o(1), \;\quad i=1,2,\dotsc,\fl{n/2}-1\\
 (S_n^\varen)_{1,1}&= \begin{cases} o(\varen), &\text{$n$ is odd} \\  o(1),&\text{$n$ is even.}
 \end{cases} 
\end{aligned} \ee
\end{proposition}
\begin{proof}
From \eqref{recs} we have this recursion:  
\begin{equation}
T^{n,n}(W^\varen_n) = \rho^n_n\circ (\rho^n_{n-1}\circ\dotsm\circ\rho^n_1) \begin{pmatrix} \left[ R^{n,n-1}_n \begin{pmatrix} T^{n-1,n-1}(W^\varen_{n-1}) \\ 
w_{n1}\ \ldots\ w_{n,n-1} \end{pmatrix} \right]^t \,\\ w_{1n}\ \ldots\ w_{n,n-1} \,\,\,\varen \end{pmatrix} .
\end{equation}
Symmetry of $W^\varen_n$ makes $T^{n,n}(W^\varen_n)$ also symmetric. 
Since   $\rho^n_n$ alters only diagonal elements,   
  the matrix must be symmetric just before the last application
of $\rho^n_n$.  
The mappings $\rho^n_{n-1}\circ\dotsm\circ\rho^n_1$ alter only entries strictly 
below the diagonal. Consequently we can skip the steps 
$\rho^n_{n-1}\circ\dotsm\circ\rho^n_1$ if we simply take the upper triangular
part of the   matrix
just before and extend it to a  symmetric matrix.  We insert one extra transposition
and then keep the lower triangular instead of the upper triangular part. 
 In other words,  let 
 \begin{eqnarray*}
W'=R^{n,n-1}_n \begin{pmatrix} T^{n-1,n-1}(W^\varen_{n-1}) \\ 
w_{n1}\ \ldots\ w_{n,n-1} \end{pmatrix}   \qquad \qquad    \text{(an $n\times(n-1)$ matrix)}  
\end{eqnarray*}
and define the  symmetric matrix $\tilde{W}=\{\tilde{w}_{ij},\, 1\leq i,j\leq n\}$
by 
  $\tilde{w}_{ij}=w'_{ij}$ for $1\leq j\leq i\wedge (n-1) $ and 
   $\tilde{w}_{nn} =\varen$.  Then $T^{n,n}(W^\varen_n) =\rho^n_n(\tilde W)$.  
In particular, the part of $T^{n,n}(W^\varen_n)$ strictly below the diagonal is
already present in $W'$.  

We prove \eqref{loteq}  by  induction on $n$. Case  $n=2$  begins with
$W_2^\tr=(w_{21})$, from which 
\[   \Lambda^\varen_2\circ T^\did_2(W^\tr_2)= 
\begin{pmatrix}  \frac{1}{2}w_{21} & \varen w_{21}\\
\varen w_{21} & 2\varen^2 w_{21}  \end{pmatrix}
=  T^{2,2} \begin{pmatrix}  \varen &  w_{21}\\
  w_{21} & \varen   \end{pmatrix}  =  T^{2,2}(W^\varen_{2}). 
\]

Assume that
$$
T^{n-1,n-1}(W^\varen_{n-1})=\Lambda^\varen_{n-1}\circ T^\did_{n-1}(W^\tr_{n-1})
+S_{n-1}^\varen.  
$$
Abbreviate  $T^\varen=(t^\varen_{ij},\,1\leq i,j\leq n-1)=T^{n-1,n-1}(W^\varen_{n-1})$
so that the induction assumption reads: 
\begin{eqnarray*}
t^\varen_{ij}&=&\varen w^{\did,n-1}_{ij}+o(\varen),\qquad\qquad\qquad\qquad i\neq j\\
t^\varen_{n-2i-1,n-2i-1}&=&2\varen^2 w^{\did,n-1}_{n-2i-1,n-2i-1}+o(\varen^2),\qquad\quad\,\, i=0,1,\dotsc\\
t^\varen_{n-2i,n-2i}&=& \frac{1}{2}w^{\did,n-1}_{n-2i,n-2i}+o(1),\qquad\qquad\qquad\,\, i=1,\dotsc\\
t^\varen_{11}&=&\varen w^{\did,n-1}_{11}+o(\varen),\qquad\qquad\qquad\quad\,\,\,\,\, \text{if}\,\, (n-1) \,\,\text{is odd}\\
                         &=&\frac{1}{2}w^{\did,n-1}_{11}+o(1),\qquad\qquad\qquad\quad\,\,\,\, \text{if}\,\, (n-1) \,\,\text{is even}\\
\end{eqnarray*}

We now perform the mapping  
$$
W' =
\rho^n_{n-1}\circ\cdots\circ \rho^n_1\begin{pmatrix} T^{n-1,n-1}(W^\varen_{n-1}) \\ 
w_{n1}\ \ldots\ w_{n,n-1} \end{pmatrix} 
$$
inductively.  
 Assume that we have applied the transformations  
$$
\rho^n_{k-1}\circ\cdots \circ \rho^n_1,\qquad k<n-1,
$$
and this has resulted in output entries
$$w'_{ij}=\varen w^{\did,n}_{ij}+o(\varen), \qquad 1\leq j < k-(n-i),\,\, n-k+1< i\leq n,
$$
where $w^{\did,n}_{ij}$ denotes the entries of the matrix $T^\did_n(W^\tr_n)$ (recall that the lower triangular part of $T^\did_n(W^\tr_n)$ is identical to $T^\tr_n(W^\tr_n)$). This is readily checked when $k-1=1$. We will show that this is also true after the transformation $\rho^n_k$. To this end, using the relations \eqref{bk} and \eqref{bk1}, we have that
\begin{eqnarray*}
w'_{nk}&=&w_{nk}(w'_{n,k-1}+t^\varen_{n-1,k})=\varen w_{nk}
(w^{\did,n}_{n,k-1}+w^{\did,n-1}_{n-1,k}) +o(\varen)
=\varen w^{\did,n}_{n,k}+o(\varen),\\
 w'_{n-j,k-j}&=&\frac{w'_{n+1-j,k-j}\,t^\varen_{n-j,k-j+1}}{t^\varen_{n-j,k-j}}
                                  \frac{w'_{n-j,k-j-1}+t^\varen_{n-j-1,k-j}}{w'_{n+1-j,k-j}+t^\varen_{n-j,k-j+1}}\\
                   &=& \frac{(\varen w^{\did,n}_{n+1-j,k-j}+o(\varen))\,(\varen w^{\did,n-1}_{n-j,k-j+1}+o(\varen))}{\varen w^{\did,n-1}_{n-j,k-j}+o(\varen)}
                  \,\,  \frac{\varen(w^{\did,n}_{n-j,k-j-1}+w^{\did,n-1}_{n-j-1,k-j})+o(\varen)}{\varen(w^{\did,n}_{n+1-j,k-j}+w^{\did,n-1}_{n-j,k-j+1})+o(\varen)}\\
                  &=& \varen \frac{w^{\did,n}_{n+1-j,k-j}\,w^{\did,n-1}_{n-j,k-j+1}}{w^{\did,n-1}_{n-j,k-j}}\,\,
                      \,\, \frac{w^{\did,n}_{n-j,k-j-1}+w^{\did,n-1}_{n-j-1,k-j}}{w^{\did,n}_{n+1-j,k-j}+w^{\did,n-1}_{n-j,k-j+1}}+o(\varen)\\
                      &=& \varen w^{\did,n}_{n-j,k-j}+o(\varen),    
\end{eqnarray*}
and this verifies the proposition for the above entries. The next step is to confirm that  
$w'_{n-j,n-j-1}=\varen w^{\did,n}_{n-j,n-j-1}+o(\varen)$  for $ j=0,\dotsc,n-2$. To this end, assume that we have performed the transformations $\rho^n_{n-2}\circ\cdots \circ\rho^n_{1}$ and  then we operate with  $\rho^n_{n-1}$. First for $j=0$, 
\begin{eqnarray*}
w'_{n,n-1}&=&w_{n,n-1}(w'_{n,n-2}+t^\varen_{n-1,n-1})\\
                  &=& w_{n,n-1}(\varen w^{\did,n}_{n,n-2}+2\varen^2 w^{\did,n-1}_{n-1,n-1}+o(\varen))\\
                  &=& \varen w_{n,n-1}w^{\did,n}_{n,n-2}+o(\varen)\\
                  &=& \varen w^{\did,n}_{n,n-1}+o(\varen).
\end{eqnarray*}
For $j>0$
\begin{eqnarray*}
w'_{n-j,n-j-1}= \frac{w'_{n+1-j,n-j-1}\,t^\varen_{n-j,n-j}}{t^\varen_{n-j,n-j-1}}\,
                          \frac{w'_{n-j,n-j-2}+t^\varen_{n-j-1,n-j-1}}{w'_{n+1-j,n-j-1}+t^\varen_{n-j,n-j}}.
\end{eqnarray*} 
To develop this  further we   distinguish between    odd and even $j$. For 
even $j$, 
\begin{eqnarray*}
w'_{n-j,n-j-1}&=&\frac{(\varen w^{\did,n}_{n+1-j,n-j-1}+o(\varen) )\,\, (\frac{1}{2} w^{\did,n-1}_{n-j,n-j}+o(1)) }{\varen w^{\did,n-1}_{n-j,n-j-1}+o(\varen)}\\
&&\quad\times \frac{\varen w^{\did,n}_{n-j,n-j-2}+o(\varen) + 2\varen^2 w^{\did,n-1}_{n-j-1,n-j-1}+o(\varen^2)}{\varen w^{\did,n}_{n+1-j,n-j-1}+o(\varen)+\frac{1}{2} w^{\did,n-1}_{n-j,n-j}+o(1)}\\
&=&\varen \frac{w^{\did,n}_{n+1-j,n-j-1}\,\,w^{\did,n}_{n-j,n-j-2}}{w^{\did,n-1}_{n-j,n-j-1}}+o(\varen)\\
&=& \varen w^{\did,n}_{n-j,n-j-1}+o(\varen)
\end{eqnarray*}
where the last step came from \eqref{bii-1}.  In the odd case 
\begin{eqnarray*}
w'_{n-j,n-j-1}&=&\frac{(\varen w^{\did,n}_{n+1-j,n-j-1}+o(\varen) )\,\, (2\varen^2 w^{\did,n-1}_{n-j,n-j}+o(\varen^2)) }{\varen w^{\did,n-1}_{n-j,n-j-1}+o(\varen)}\\
&&\quad\times \frac{\varen w^{\did,n}_{n-j,n-j-2}+o(\varen) + \frac{1}{2} w^{\did,n-1}_{n-j-1,n-j-1}+o(1)}{\varen w^{\did,n}_{n+1-j,n-j-1}+o(\varen)+2\varen^2 w^{\did,n-1}_{n-j,n-j}+o(\varen^2)}\\
&=&\varen \frac{w^{\did,n-1}_{n-j,n-j}\,\,w^{\did,n-1}_{n-j-1,n-j-1}}{w^{\did,n-1}_{n-j,n-j-1}}+o(\varen)\\
&=& \varen w^{\did,n-1}_{n-j,n-j-1}+o(\varen) =  \varen w^{\did,n}_{n-j,n-j-1}+o(\varen).
\end{eqnarray*}
The second last equality follows from the fact that $T^\did_{n-1}(W_{n-1}^\tr)$ satisfies \eqref{collapse} with $n$  replaced by $n-1$.  The last equality 
comes from  the definition of  $b^{\triangle,n}_{n-j,n-j-1}$ as the identity
mapping (see the bullet  below \eqref{bii-1}).  In the case $(n-j,n-j-1)=(2,1)$ we need to distinguish between the case $n$ is even or odd. In the even case we have
\begin{eqnarray*}
w'_{21}&=&t^\varen_{11}\frac{w'_{31}t^\varen_{22}}{t^\varen_{21}(w'_{31}+t^\varen_{22})}\\
&=&(\varen w^{\did,{n-1}}_{11}+o(\varen))\frac{(\varen w^{\did,n}_{31}+o(\varen))(\frac{1}{2}w^{\did,n-1}_{22}+o(1))}{(\varen w^{\did,n-1}_{21}+o(\varen))(\varen w^{\did,n}_{31}+o(\varen)+\frac{1}{2}w^{\did,n-1}_{22}+o(1) ) }\\
&=& \varen w^{\did,{n-1}}_{11} \frac{ w^{\did,n}_{31}}{w^{\did,n-1}_{21}}+o(\varen)= \varen\frac{ w^{\did,n}_{31}}{w^{\did,n-1}_{21}}+o(\varen) =\varen w^{\did,n}_{21}+o(\varen),
\end{eqnarray*}
where the second to last equality follows from \eqref{collapse}, since $(n-1)$ is odd and therefore $w^{\did,{n-1}}_{11} =1$.
The case that $n$ is odd follows similarly.

To complete the construction of $T^{n,n}(W^\varen_n)$,  extend $W'$
to the symmetric matrix  $\tilde W$ as explained above 
and define $W'' =\rho^n_n(\tilde W)$. 
By  computations similar to the ones above  and by  symmetry,   
the  diagonal elements $(w''_{ii})_{i=1,\dotsc,n}$ satisfy 
 $w''_{n-2k,n-2k}=2\varen^2  w^{\did,n}_{n-2k,n-2k-1}$ and $w''_{n-2k-1,n-2k-1}=\frac{1}{2}w^{\did,n}_{n-2k,n-2k-1}$ for $k=0,1,\dotsc$.
The proof is then complete.
\end{proof}

\begin{proof}[Proof of Theorem  \ref{shape-of-tr}] 
Consider a symmetric, $n\times n$, matrix, $W^\varen_n$, with diagonal weights, $w_{ii}=\varen,\,\,i=1,2,\dotsc,n$. 
 Let $v_r$ denote the partition sum introduced in \eqref{path1} 
with $k=m=n$: 
\be v_r = \sum_{(\pi_1,\ldots,\pi_r)\in\Pi^{(r)}_{n,n}} \prod_{(i,j)\in \pi_1\cup\cdots\cup\pi_r} w_{ij}.  
\label{v1}\ee
The key observation is the following. For $1\le k\le \fl{n/2}$, 
\be\begin{aligned}
v_{2k}= \Bigl( \;\prod_{i=1}^k w_{ii} w_{n-i+1, n-i+1}\Bigr) z_k^2 + V(2k+1)
\end{aligned}\label{vz1}\ee
and 
\be\begin{aligned}
v_{2k-1}= \Bigl( \;\prod_{i=1}^k w_{ii} w_{n-i+1, n-i+1}\Bigr) 2z_{k-1}z_k + V(2k+1)
\end{aligned}\label{vz2}\ee
where $z_0=1$, $z_r$ is defined by \eqref{d:Z1},  and  the unspecific notation $V(\ell)$ represents any sum of products 
of weights where each term
contains  at least $\ell$ diagonal weights $w_{ii}$.  

To see the origin of  \eqref{vz1}--\eqref{vz2}, 
  consider first $v_1$, the sum of products $\prod_{(i,j)\in \pi} w_{ij}$
over all paths $\pi$ from $(1,1)$ to $(n,n)$.  Those products that 
contain only weights $w_{11}w_{nn}$ from the diagonal
correspond to paths that   stay either strictly 
above or strictly below the diagonal, except at points $(1,1)$ and  $(n,n)$. 
By the symmetry of the weights this gives two copies of $z_1$.  
Similarly for $v_2$,  pairs $(\pi_1,\pi_2)$ that intersect the diagonal
only at $\{(1,1), (n,n)\}$ correspond to pairs such  that $\pi_2$ 
connects $(1,2)$ to $(n-1,n)$ above the diagonal and  $\pi_1$ connects  
$(2,1)$ to $(n,n-1)$ below the diagonal.  Weights of paths are multiplied,
and so symmetry gives $z_1^2$.  The higher cases work the same way.

For the symmetric weight matrix the shape vector
$x=(x_1,\dotsc, x_n)$ is given 
by
\be  x_1=v_1, \quad x_i= z_{n,i}= \frac{v_i}{v_{i-1}} \quad\text{for $2\le i\le n$.} 
\label{d:x1}\ee 
Here we recalled that the shape vector is the bottom row $z_{n\cdot}$ of the $P$ pattern, see \eqref{Ppattern}, and combined \eqref{path1}  with \eqref{tz1}.  

 Since $w_{ii}=\varen$, equations \eqref{vz1}--\eqref{d:x1} 
 combine to give the following asymptotics for
   $k=1,2,\dotsc,\fl{n/2}$  as $\varen\searrow 0$:
 \begin{eqnarray*}
 x_{1}&=&\,\,\,\,v_1\,\,\,\,\,=2\varen^2 z_1+o(\varen^2),\\
 &&\\
 x_{2k}&=&\frac{v_{2k}}{v_{2k-1}}=\frac{\varen ^{2k} z_k^2+o(\varen^{2k})}{\varen^{2k}\,\,2z_{k-1}z_k+o(\varen^{2k})}
 =\frac{1}{2}\frac{z_k}{z_{k-1}}+o(1),\\
 &&\\
 x_{2k+1}&=&\frac{v_{2k+1}}{v_{2k}}=\frac{\varen ^{2(k+1)} \,\,2z_kz_{k+1}+o(\varen^{2(k+1)})}{\varen^{2k}\,z_k^2+o(\varen^{2k})}
 =2\varen^2\,\,\frac{z_{k+1}}{z_k}+o(\varen^2).
 \end{eqnarray*}

The proof can be now completed by comparing to \eqref{loteq} and using \eqref{collapse}.
\end{proof}

For a triangular array $X=(x_{ij},\,1\leq j< i\leq n)\in (\mathbb{R}_{>0})^{n(n-1)/2}$ we define its type, $\tau=(\tau^n_j)_{0\le j\le n-1}=\text{type} \,X$, as the vector  with
entries 
\begin{eqnarray*}
\tau^n_{j}(X)=\frac{D_{nj}(X)}{D_{n,j-1}(X)}    
\end{eqnarray*}
where
\begin{eqnarray*}
D_{n0}(X)&=&1\\
\text{and}\qquad 
D_{nj}(X)&=&x_{nj}x_{n-1,j-1}\cdots x_{n-j+1,1}  ,\qquad j=1,2,\dotsc,n-1.
\end{eqnarray*}
\begin{proposition}\label{proptr}
Let  $W_n=(w_{ij},\,1\leq j< i\leq n)$ with $w_{ij}\in \mathbb{R}_{>0}$. We have
\begin{eqnarray}\label{trtype}
\tau^n_{j}(T^\tr_n(W_n))= \prod_{\ell=1}^{j-1} w_{j,\ell} \prod_{k=j+1}^{n} w_{kj},\qquad 1\leq j\leq n-1.
\end{eqnarray}
\end{proposition}
\begin{proof}
Let us first notice that if $X=(x_{ij},\,1\leq j < i \leq n)$ is a triangular array and $X'=\rho^{\tr,n}_j(X)$, then
\begin{eqnarray}\label{rifrac}
\frac{x'_{nj}\cdots x'_{n-j+1,1}}{x'_{n,j-1}\cdots x'_{n-j+2,1}}=x_{nj}\, \frac{x_{n-1,j}\cdots x_{n-j,1}}{x_{n-1,j-1}\cdots x_{n-j+1,1}},\qquad j< n-1.
\end{eqnarray}
To check this we notice that $\rho^{\tr,n}_j=\rho^n_j=h_j\circ r_j$, where $h_j$ and $r_j$ are defined in \eqref{bk} via the Bender-Knuth transformations. Let us recall that
\begin{eqnarray}\label{bkrecall}
(b_{ij}(X))_{ij}=x'_{ij}=\frac{x_{i+1,j}\,x_{i,j+1}}{x_{ij}}\frac{x_{i,j-1}+x_{i-1,j}}{x_{i+1,j}+x_{i,j+1}},
\end{eqnarray}
with the same convention as in \eqref{bk1}. Multiplying the various relations \eqref{bkrecall} for $(i,j)=(n,j),\dotsc,(n-j+1,1)$ leads to \eqref{rifrac}. Iterating this leads to
\begin{eqnarray}\label{rifrac2}
\tau^n_{j}(T^\tr_n(W_n))
&=&w_{n,j} \,\,\tau^{n-1}_{j}(T^\tr_{n-1}(W_{n-1})) \nonumber\\
&=&w_{n,j}\cdots w_{j+2,j} \,\,\,\tau^{j+1}_{j}(T^\tr_{j+1}(W_{j+1})).
\end{eqnarray}
Denoting by $w'_{ij}$ the elements of $T^\tr_{j+1}(W_{j+1})$, by $w^\dagger_{ij}$ the elements of 
$T^\tr_{j}(W_{j})$ and using the transformations in \eqref{bii-1} we obtain that
\begin{eqnarray*}
\tau^{j+1}_{j}(T^\tr_{j+1}(W_{j+1}))&=&\frac{w'_{j+1,j}\cdots w'_{21}}{w'_{j+1,j-1}\cdots w'_{31}}\\
&=&w_{j+1,j}\,\,\frac{\prod_{\ell=0}^{\fl{j/2}-1} w^\dagger_{j-2\ell,j-2\ell-1} }{\prod_{\ell=0}^{\fl{(j-1)/2}-1}w^\dagger_{j-2\ell-1,j-2\ell-2}}.  
\end{eqnarray*}
Using Theorem \ref{shape-of-tr} we have that 
$$  \prod_{\ell=0}^{\fl{j/2}-1}w^\dagger_{j-2\ell,j-2\ell-1} =\prod_{1\leq \ell<k \leq j}w_{k\ell}.
$$
The  definition below  \eqref{bii-1} implies that $w^\dagger_{j-2\ell-1,j-2\ell-2} =T^\tr_{j-1}(W_{j-1})_{(j-1)-2\ell,(j-1)-2\ell-1}$  for $\ell=0,\dotsc,\fl{(j-1)/2}-1$ and using again Theorem \ref{shape-of-tr} we obtain 
$$\prod_{\ell=0}^{\fl{(j-1)/2}-1} w^\dagger_{j-2\ell-1,j-2\ell-2} = \prod_{1\leq \ell < k \leq j-1} w_{k\ell}.
$$
Combining the last three relations gives 
$$\tau^{j+1}_{j}(T^\tr_{j+1}(W_{j+1}))= w_{j+1,j} \prod_{1\leq \ell < j} w_{j\ell},
$$
and this completes the proof.
\end{proof}

By  combining Theorems \ref{thmtr2} and \ref{thmtr1} and Proposition \ref{proptr} we   identify  the probability distribution of the shape vector of the
triangular array under inverse gamma weights.  
The mapping that gives the shape vector is  $\sigma^\tr: (\R_{>0})^{n(n-1)/2} \to (\R_{>0})^{n-1}$ defined by
\begin{eqnarray} 
\sigma^\tr(W)=\mbox{sh }\,T^\tr_n(W)&=&(t_{n,n-1},t_{n-1,n-2},\ldots, t_{2,1} ).
\end{eqnarray}
Consider the probability measure 
\begin{eqnarray}\label{weighttr}
\lambda_{\alpha}(dw) = Z_\alpha^{-1}\prod_{1\leq j<i\leq n} w_{ij}^{-\alpha_i-\alpha_j} 
\exp\Bigl(-\sum_{1\leq j<i\leq n} \frac1{w_{ij}} \;\Bigr) \prod_{1\leq j< i\leq n} \frac{dw_{ij}}{w_{ij}}  
\end{eqnarray}
on the space of triangular arrays $(w_{ij},\,1\leq j < i\leq n)\in(\R_{>0})^{n(n-1)/2}$, where $\alpha=(\alpha_1,\dotsc,\alpha_n)$,   $\alpha_i+\alpha_j>0$ and   the normalisation is
\[  Z_\alpha=   \prod_{1\le j<i\le n} \Gamma(\alpha_i+\alpha_j). \]

\begin{corollary}\label{pushtr} 
For the $\lambda_\alpha$-distributed triangular array of weights, the distribution of the shape vector is given by  
\begin{eqnarray*}
&&\lambda_{\alpha}\circ (\sigma^\tr)^{-1}(dt) \\
&=&  \prod_{1\le j<i\le n} \Gamma(\alpha_i+\alpha_j)^{-1} 
 \left(\frac{\prod_{\ell=0}^{\fl{\frac{n-1}{2}}-1 }t_{n-2\ell-1,n-2\ell-2}}{\prod_{\ell=0}^{\fl{\frac{n}{2}}-1}t_{n-2\ell,n-2\ell-1} }\right)^{\alpha_n} e^{-\,\frac1{t_{2,1}}} \Psi^{n-1}_{\alpha'}(t)
 \prod_{0\leq i \leq n-2} \frac{dt_{n-i,n-i-1}}{t_{n-i,n-i-1}} 
\end{eqnarray*}
where  $\alpha'=(\alpha_1,\dotsc,\alpha_{n-1})$. 
\end{corollary}
\begin{proof}
Let  $T=(t_{ij},\,\,1\leq j< i \leq n)=T^\tr_n(W)$. 
We convert the density \eqref{weighttr} into $t_{ij}$ variables.  By Proposition \ref{proptr}, 
\begin{eqnarray}\label{cor66eq1}
\prod_{j<i} w_{ij}^{-\alpha_i-\alpha_j} 
&=& \prod_{j=1}^{n}\left(\;\prod_{\ell=1}^{j-1} w_{j\ell}  \cdot \prod_{k=j+1}^n w_{kj}\right)^{-\alpha_j}
= \prod_{j=1}^{n-1}(\tau^n_j)^{-\alpha_j}  \cdot 
\left(\;\prod_{j=1}^{n-1} w_{nj}\right)^{-\alpha_n}. \nonumber
\end{eqnarray}
From the proof of Proposition \ref{proptr} (after relation \eqref{rifrac2}),  
\begin{eqnarray*}
\prod_{\ell=0}^{\fl{n/2}-1} t_{n-2\ell,n-2\ell-1}&=&\prod_{1\leq j<i\leq n} w_{ij}\\
\text{and} \qquad \prod_{\ell=0}^{\fl{(n-1)/2}-1} t_{n-2\ell-1,n-2\ell-2}
&=& \prod_{1\leq j<i\leq n-1} w_{ij}.
\end{eqnarray*}
Combine these with  Theorem \ref{thmtr2} to obtain 
\begin{eqnarray*}
\prod_{j<i} w_{ij}^{-\alpha_i-\alpha_j} 
\exp\Bigl(-\sum_{j<i} \frac1{w_{ij}} \;\Bigr) =
 \left(\frac{\prod_{\ell=0}^{\fl{\frac{n-1}{2}}-1 }t_{n-2\ell-1,n-2\ell-2}}{\prod_{\ell=0}^{\fl{\frac{n}{2}}-1}t_{n-2\ell,n-2\ell-1} }\right)^{\alpha_n}   \,\,
 \prod_{j=1}^{n-1}(\tau^n_j)^{-\alpha_j} \,\,e^{-\mathcal{E^\tr}(T)}.  
\end{eqnarray*}
By the volume preserving property of the
$W\mapsto T$  map (Theorem \ref{thmtr1}),   
\begin{eqnarray*}
&&\lambda_{\alpha}\circ (T^\tr)^{-1}(dt)
=  \left(\frac{\prod_{\ell=0}^{\fl{\frac{n-1}{2}}-1 }t_{n-2\ell-1,n-2\ell-2}}{\prod_{\ell=0}^{\fl{\frac{n}{2}}-1}t_{n-2\ell,n-2\ell-1} }\right)^{\alpha_n}  \,\,\prod_{j=1}^{n-1}(\tau^n_j)^{-\alpha_j} \,\, e^{-\mathcal{E^\tr}(T)} \prod_{j<i} \frac{dt_{ij}}{t_{ij}}.
\end{eqnarray*}
The result then follows by integrating over the variables $(t_{ij},\,\,1\leq j<i-1,\,\,1\leq i \leq n)$ and the definition of the Whittaker function.
\end{proof}

As a further corollary we record the distribution of the vector 
$(z_1, z_2/z_1, \dotsc , z_{\fl{n/2}}/z_{\fl{n/2}-1})$ of ratios of partition functions $z_r$   defined by  \eqref{parttr}.   The result comes by 
 combining Corollary \ref{pushtr} with Theorem \ref{shape-of-tr}.  
\begin{corollary}
Let the array of weights $(w_{ij},\,1\leq j < i\leq n)$ have distribution 
$\lambda_\alpha$ of \eqref{weighttr}, and as before
$\alpha=(\alpha_1,\dotsc,\alpha_n)=(\alpha', \alpha_n)$.
Then the distribution of the vector $(z_1, z_2/z_1, \dotsc , z_{\fl{n/2}}/z_{\fl{n/2}-1})$, with the partition functions $z_r$  defined in \eqref{parttr}, is given as follows
in terms of the integral of a bounded Borel function $\varphi$: 
\be \label{z-distr}\begin{aligned} 
&\int \varphi(z_1, z_2/z_1, \dotsc , z_{\fl{n/2}}/z_{\fl{n/2}-1})\,\lambda_\alpha(dw) \\
&=\prod_{1\le j<i\le n} \!\!\Gamma(\alpha_i+\alpha_j)^{-1}   
\!\! \int\limits_{(\R_{>0})^{\fl{\frac{n}2}}} \prod_{\substack{0\le i\le n-2:\\[1pt] \text{$i$ even}}} \frac{dt_{n-i,n-i-1}}{t_{n-i,n-i-1}}  
\;\varphi( t_{n,n-1}, t_{n-2,n-3}, \dotsc, t_{n-2\fl{\frac{n}2}+2, n-2\fl{\frac{n}2}+1})   \\[3pt]
&\quad \times \int\limits_{(\R_{>0})^{\ce{n/2}-1}} 
\, \left(\frac{\prod_{k=0}^{\fl{\frac{n-1}{2}}-1 }t_{n-2k-1,n-2k-2}}{\prod_{k=0}^{\fl{\frac{n}{2}}-1}t_{n-2k,n-2k-1} }\right)^{\alpha_n}  e^{-\,\frac1{t_{2,1}}} \Psi^{n-1}_{\alpha'}(t)
 \prod_{\substack{1\le i\le n-2:\\[1pt] \text{$i$ odd}}} \frac{dt_{n-i,n-i-1}}{t_{n-i,n-i-1}}. 
\end{aligned}\ee
\end{corollary}

The results above are related to those of symmetric weight matrices in
several  ways. 

(i)    Replace $n$ with
$n-1$ in Corollary \ref{pfs} and 
consider a  symmetric $(n-1)\times(n-1)$  weight matrix with distribution \eqref{nutilde},  and set     $\dipa=\alpha_n$.   Let $\sigma_1=t_{n-1,n-1}$ be
the polymer partition function of the symmetric matrix, 
or equivalently, the front element  of its shape vector.  
Then  a comparison of  \eqref{z-distr}  and \eqref{nu4}  reveals that   the  distribution of the partition function $z_1$ is
 identical to the distribution of $2t_{n-1,n-1}$.  

(ii) Corollary \ref{pushtr}  can be obtained as the $\dipa\to\infty$ limit of Corollary \ref{pfs}.  Using the recursive structure \eqref{wdef} of Whittaker functions, namely
  $\Psi^n_\alpha  = Q^{n,n-1}_{\alpha_n}\Psi^{n-1}_{\alpha'}$,  one can show that
$$\tilde\nu_{\alpha,\dipa}\circ\sigma^{-1}  \,\,\,\Longrightarrow \,\,\,\lambda_{\alpha}\circ (\sigma^\tr)^{-1}  \qquad 
\text{as}\,\,\, \dipa\to\infty,
$$  
where ``$\Longrightarrow$" denotes weak convergence of probability measures.  
Under the measure $\tilde\nu_{\alpha,\dipa}$ the diagonal element $w_{ii}$ of the symmetric input matrix has probability distribution 
\[   \rho_{ii}(du)=\frac{ u^{-\alpha_i-\dipa} \, e^{-1/(2u)}}{2^{\alpha_i+\dipa}\Gamma(\alpha_i+\dipa) } \cdot \frac{du}{u} \quad \text{on $0<u<\infty$},  \]
and hence its reciprocal $w_{ii}^{-1}$ is twice a gamma variable 
with parameter $\alpha_i+\dipa$.   Consequently 
$\dipa w_{ii}\to 1/2$  almost everywhere as $\dipa\to\infty$.    
 Thus $w_{ii}$ decays as $(1/2)\dipa^{-1}$.  This corresponds to the appearance,
 in our proof, of  triangular arrays with diagonal elements  $\varen\to 0$. 
 
(iii)   The limit $\dipa\to\infty$, or equivalently $\varen\to 0$, introduces a {\sl depinning} 
 effect on the polymer, which is responsible for the appearance of the hard wall phenomenon. It is also worth noting   that the structure of the measure $\lambda_{\alpha}\circ (\sigma^\tr)^{-1}$ is similar to that of $\tilde\nu_{\alpha,\dipa}\circ\sigma^{-1} $. In other words it appears that the hard wall produces also a ``pinning" effect along the line $\{(j,n),\,\,\,1\leq j\leq n\}$. 

\section{Whittaker integral identities}\label{wid}

In this section, we recall three integral identities for Whittaker functions which were proved in the 
papers~\cite{stade,stade-jams}, and explain how they are equivalent to (and in fact generalized by)
those which have appeared naturally in the context of the present paper (Corollaries \ref{cor-bs-gen},
\ref{stade} and \ref{nwid}).
We first note that the functions $W_{n,a}(y)$ introduced in Section \ref{wp} are denoted by $W_{n,2a}(y)$ 
in the papers~\cite{stade,stade-jams}.  The following identity was conjectured by Bump~\cite{bump}
and proved by Stade~\cite[Theorem 1.1]{stade}.  
\begin{thm}[Stade]\label{st1} 
For $s\in\C$, $a,b\in\C^n$ with $\sum_i a_i = \sum_i b_i=0$,
\begin{align}\label{stadei}
 \int_{(\R_{>0})^{n-1} } W_{n,a}(y) W_{n,b}(y) 
\prod_{j=1}^{n-1} (\pi y_j)^{2js} & (2 y_j^{-j(n-j)}) \frac{dy_j}{y_j} \nonumber \\
&= \Gamma(ns)^{-1} \prod_{j,k} \Gamma(s+a_j+b_k) .\end{align}
\end{thm}
This integral is associated, via the Rankin-Selberg method, with Archimedean $L$-factors of automorphic 
$L$-functions on $GL(n,\R)\times GL(n,\R)$. 
Using (\ref{rel}), it is straightforward to see that this is equivalent to:
\begin{thm}[Stade]\label{thm-stade-psi} 
Suppose $r>0$ and $\lambda,\nu\in\C^n$.  Then
\begin{equation}\label{stade-psi}
\int_{(\R_{>0})^n } e^{-r x_1} \Psi^n_{-\nu}(x) \Psi^n_{-\lambda}(x)\prod_{i=1}^n \frac{dx_i}{x_i}
= r^{-\sum_{i=1}^n (\nu_i+\lambda_i)} \prod_{ij}\Gamma(\nu_i+\lambda_j).
\end{equation}
\end{thm}
Indeed, if we let $$a_j=\l_j-(1/n)\sum_i\l_i,\qquad b_j=\nu_j-(1/n)\sum_i\nu_i$$ and $s=(1/n)\sum_i (\l_i+\nu_i)$
then, using (\ref{rel}) and (\ref{shift}), (\ref{stadei}) becomes
$$\Gamma(ns) \int_{(\R_{>0})^{n-1} } \Psi^n_{-\nu}(x) \Psi^n_{-\lambda}(x) x_1^{-ns} 
2^{n-1} \prod_{j=1}^{n-1} \frac{dy_j}{y_j} = \prod_{ij}\Gamma(\nu_i+\lambda_j),$$
where $\pi y_j =\sqrt{x_{n-j+1}/x_{n-j}} $ for $j=1,\ldots,n-1$.  It is important to note here
that we are regarding $\Psi^n_{-\nu}(x) \Psi^n_{-\lambda}(x) x_1^{-ns}$ as a function of
$y_1,\ldots,y_{n-1}$.  Now, writing
$$\Gamma(ns)=\int_0^\infty x_1^{ns} e^{-x_1} \frac{dx_1}{x_1}$$
we can absorb this into the integral, changing variables from $y_1,\ldots,y_{n-1},x_1$
to $x_1,\ldots,x_n$, to obtain
$$\int_{(\R_{>0})^n } e^{-x_1} \Psi^n_{-\nu}(x) \Psi^n_{-\lambda}(x)\prod_{i=1}^n \frac{dx_i}{x_i}
= \prod_{ij}\Gamma(\nu_i+\lambda_j).$$
The identity (\ref{stade-psi}) follows, using (\ref{a}).

The second identity is a formula for the Mellin transform
\begin{align*}
N_{n,b,a}(s)=\int_{(\R_{>0})^{n-1} } & W_{n,a}(y_1,\ldots,y_{n-1}) W_{n-1,b}(y_1,\ldots,y_{n-2}) \\
& \prod_{j=1}^{n-1} (\pi y_j)^{2js} (2 y_j^{-j(n-j-1/2)}) \frac{dy_j}{y_j} ,\end{align*}
for $s\in\C$, $n\ge 3$ and $a\in\C^n$, $b\in\C^{n-1}$ with $\sum_i a_i=\sum_j b_j=0$.
This integral is associated with archimedean $L$-factors of automorphic 
$L$-functions on $GL(n-1,\R)\times GL(n,\R)$. The following identity was conjectured by Bump~\cite{bump}
and proved by Stade~\cite[Theorem 3.4]{stade-jams}. 
\begin{thm}[Stade]\label{st2}
$$N_{n,b,a}(s)=\prod_{i,j}\Gamma(s+a_i+b_j).$$
\end{thm}
Now, for $\lambda\in\C^n$ and $r>0$,
$$\Psi^{n-1}_{\lambda;r}(x_1,\ldots,x_{n-1})=r^{\lambda_n} \Psi^n_\lambda(x_1,\ldots,x_{n-1},r).$$
Using this, and the relations (\ref{rel}) and (\ref{inv}), it is straightforward to see that Theorem \ref{st2} is equivalent to:
\begin{thm}[Stade]\label{thm-stade2-psi}
Let $r>0$, $\lambda\in\C^{n-1}$ and $\nu\in\C^n$.  Then
\begin{equation}\label{bs-gen1} 
\int_{(\R_{>0})^{n-1} } \Psi^{n-1}_{\nu;r}(x) \Psi^{n-1}_\lambda(x) \prod_{i=1}^{n-1} \frac{dx_i}{x_i}
= r^{-\sum_{i=1}^{n-1} (\nu_i+\lambda_i)} \prod_{ij}\Gamma(\nu_i+\lambda_j) .
\end{equation} 
\end{thm}

The third identity is a formula for the Mellin transform
$$M_{n,a}(s) = \int_{(\R_{>0})^{n-1} } W_{n,a}(y) 
\prod_{j=1}^{n-1} (\pi y_j)^{2s_j} (2 y_j^{-j(n-j)/2}) \frac{dy_j}{y_j} ,$$
for particular values of $s=(s_1,\ldots,s_{n-1})$ lying on a two-dimensional
subspace of $\C^{n-1}$.   This integral is associated
with an archimedean $L$-factor of an exterior square automorphic $L$-function on $GL(n,\R)$.
The following identity was conjectured by Bump and Friedberg~\cite{bf} 
and proved by Stade~\cite[Theorem 3.3]{stade-jams}. 
\begin{thm}[Stade]\label{st3} 
Let $s_1,s_2\in\C$ and $a\in\C^n$ with $\sum_i a_i=0$.  
Suppose that, for $2<j\le n-1$, $s_j=\epsilon(j) s_1+(j-\epsilon(j)) s_2/2$, 
where $\epsilon(j)=1$ if $j$ is odd and $0$ otherwise.  
Set $s_n=\epsilon(n) s_1+(n-\epsilon(n)) s_2/2$. Then
for  $s=(s_1,\ldots,s_{n-1})$, 
\begin{equation}\label{stade2}
\Gamma(s_n) M_{n,a}(s) = \prod_i \Gamma(s_1+a_i) \prod_{i<j} \Gamma(s_2+a_i+a_j).
\end{equation}
\end{thm}
In terms of the $\Psi^n_\lambda$, this is equivalent to the following identity, which is straightforward
to verify using (\ref{rel}) and (\ref{a}) as above.
\begin{thm}[Stade]\label{thm-stade3-psi} 
Suppose $r>0$, $\lambda\in\C^n$ and $\gamma\in\C$. Then
\begin{align*}
\int_{(\R_{>0})^n} & f(x')^\gamma e^{-r x_1} \Psi^n_{-\lambda}(x) \prod_{i=1}^n \frac{dx_i}{x_i} \\
& = c_n(r,\gamma) r^{-\sum_{i=1}^n\lambda_i} 
\prod_i \Gamma(\lambda_i+\gamma) \prod_{i<j}\Gamma(\lambda_i+\lambda_j) ,
\end{align*}
where $x'_i=1/x_{n-i+1}$, $f(x)=\prod_i x_i^{(-1)^i}$ and 
$$c_n(s,\gamma)=\begin{cases} 1 & \mbox{ if $n$ is even,}\\
s^{-\gamma} & \mbox{ if $n$ is odd.}
\end{cases}$$
\end{thm}
Note that $f(x')=f(x)$ if $n$ is even and $f(x')=1/f(x)$ if $n$ is odd.

\section{Tropicalization, last passage percolation and random matrices}
\label{sec:trop} 

The geometric RSK correspondence is a geometric lifting of the (Berenstein-Kirillov extension
of the) RSK correspondence.  Going the other way, let $x^\epsilon_{ij}=e^{y_{ij}/\epsilon}$ where
$Y=(y_{ij})\in\R^{n\times m}$ and $\epsilon>0$.  Let $X^\epsilon=(x^\epsilon_{ij})$ and 
$T^\epsilon=(t^\epsilon_{ij})=T(X^\epsilon)$.  Then the mapping $U:\R^{n\times m}\to\R^{n\times m}$ 
defined by $U(Y)=(u_{ij})$ where $u_{ij}=\lim_{\epsilon\to 0} \epsilon \log t^\epsilon_{ij}$
is the extension of the RSK mapping to matrices with real entries introduced by 
Berenstein and Kirillov~\cite{bki}.  We identify the output $U(Y)$ with a pair of patterns as before,
but now the entries are allowed to take real values.  In this context, we define a {\em real pattern} 
of height $h$ and shape $x\in\R^n$ as an array of real numbers
$$R=\begin{array}{cccccccccc}
&&&r_{11}&&&&&\\
&&r_{22}&&r_{21}&&&&\\
&\iddots&&&&\ddots&&&\\
r_{nn}&&&\ldots&&&r_{n1}&&\\
&\ddots&&&&&&\ddots&\\
&&r_{hn}&&&\ldots&&&r_{h1}
\end{array}$$
with bottom row $r_{h\cdot} = x$.  The range of indices is 
$$L(n,h)=\{(i,j):\ 1\le i\le h,\ 1\le j\le i\wedge n\}.$$
Fix a real pattern $R$ as above.  Set $s_0=1$ and, for $1\le i\le h$,
$s_i=\sum_{j=1}^{i\wedge n} r_{ij}$ and $c_i=s_i-s_{i-1}$.
We shall refer to $c$ as the {\em type} of $R$ and write $c=\mbox{\small{type}}(R)$.
Denote by $\Sigma^h(x)$ the set of real patterns with shape $x$ and height $h$.
We say that a real pattern $R$ is a (generalized) Gelfand-Tsetlin pattern
if $r_{nn}\ge 0$ and it satisfies the interlacing property $r_{i+1,j+1}\le r_{ij}\le r_{i+1,j}$
for all $(i,j)\in L(n,h)$ with $i<h$, with the conventions $r_{i+1,n+1}=0$ for
$i=n,\ldots,h-1$.  Denote the set of generalized Gelfand-Tsetlin patterns
with height $h$ and shape $x\in\R_+^n$ by $GT^h(x)$.  This is a Euclidean polytope of
dimension $d=n(n-1)/2+(h-n+1)n$.  Denote the corresponding Euclidean measure
by $dR$.  The analogue of the Whittaker functions in this setting are the functions
$J_\lambda(x)$ defined, for $\lambda\in\C^h$ and $x\in\R_+^n$ by
$$J_\lambda(x)=\int_{GT^h(x)} e^{-\lambda\cdot\mbox{\tiny{type}}(R)} dR.$$
Note that, from \eqref{gf}, we have
$$\lim_{\epsilon\to 0} \epsilon^{d} \Psi^n_{\epsilon\lambda;1}( e^{x/\epsilon})=J_\lambda(x).$$
If $h=n$ then $J_\lambda(x)=\det(e^{-\lambda_i x_j})/\Delta(\lambda)$ where
$\Delta(\lambda)=\prod_{i>j} (\lambda_i-\lambda_j)$ (see, for example,~\cite{noc1}).

The analogue of Theorem~\ref{bi} in this setting is the following.  This result can be inferred 
directly from results of \cite{bki} (see Property 8 after the statement of Theorem 1.1)
or seen as a consequence of Theorem~\ref{bi}.
We identify the output $U(Y)$ with a pair of real patterns $(R,S)$ of respective
heights $m$ and $n$, and common shape $(u_{nm},\ldots,u_{n-p+1.n-p+1})$,
where $p=n\wedge m$.

\begin{cor}\label{vpt} 
The output $U(Y)=(R,S)$ is a pair of generalized Gelfand-Tsetlin patterns
if, and only if, all of the entries of $Y$ are non-negative.  
\end{cor}
We note that the corresponding statement for matrices with integer entries follows as a particular case.
If $Y$ has non-negative integer entries then the pair of generalized
Gelfand-Tsetlin patterns obtained can be interpreted in the usual way as the pair of 
semistandard tableaux obtained via the RSK correspondence.  

The Berenstein-Kirillov \cite{bki} definition of $U$ in terms of lattice paths is given as follows.
For $1\le k\le m$ and $1\le r\le n\wedge k$,
\be u_{n-r+1,k-r+1}+ \dotsm + u_{n-1,k-1} + u_{nk} 
= \max_{(\pi_1,\ldots,\pi_r)\in\Pi^{(r)}_{n,k}} \sum_{(i,j)\in \pi_1\cup\cdots\cup\pi_r} y_{ij},
\label{path2t}\ee
where $\Pi^{(r)}_{n,k}$ denotes the set of $r$-tuples of non-intersecting
directed nearest-neighbor lattice paths $\pi_1,\ldots,\pi_r$
starting at positions $(1,1),(1,2),\ldots,(1,r)$ and ending at positions $(n,k-r+1),\ldots,(n,k-1),(n,k)$. 
(See Figure \ref{fig5}.) 
This determines the entries of $R$.  The entries of $S$ are given by similar formulae using
$U(Y^t)=(S,R)$. In particular,
\begin{equation}\label{lpt}
u_{nm}=\max_{\pi\in\Pi^{(1)}_{n,m}} \sum_{(i,j)\in\pi} y_{ij},
\end{equation}
where $\Pi^{(1)}_{n,m}$ is the set of directed nearest-neighbor lattice paths in $\Z^2$ from $(1,1)$ to $(n,m)$.
This formula provides a connection to last passage directed percolation which we will discuss shortly.
The formula \eqref{path2t} is the analogue of Greene's theorem in this setting (see, for example, \cite[\S 3.1]{fulton}).

The local move description of Section \ref{gRSKsec} carries over to the tropical setting, as follows.
For convenience and clarity we adopt the same notation as in the geometric setting.
For each $2\le i\le n$ and $2\le j\le m$ define a mapping $l_{ij}$ which takes as input a matrix 
$Y=(y_{ij})\in\R^{n\times m}$ and replaces the submatrix
$$ \begin{pmatrix} y_{i-1,j-1}& y_{i-1,j}\\ y_{i,j-1}& y_{ij}\end{pmatrix}$$
of $Y$ by its image under the map
\begin{equation}\label{abcdt}
\begin{pmatrix} a& b\\ c& d\end{pmatrix} \qquad \mapsto \qquad 
\begin{pmatrix} b\wedge c-a & b\\ c& d+b\vee c \end{pmatrix},
\end{equation}
and leaves the other elements unchanged.  For $2\le i\le n$ and $2\le j\le m$, 
define $l_{i1}$ to be the mapping that replaces
the element $y_{i1}$ by $y_{i-1,1}+y_{i1}$ and $l_{1j}$ to be the mapping that replaces the element $y_{1j}$ by 
$y_{1,j-1}+y_{1j}$.  As before we define $l_{11}$ to be the identity map.
For $1\le i\le n$ and $1\le j\le m$, set
$$\pi^j_i=l_{ij}\circ\cdots\circ l_{i1},$$
and, for $1\le i\le n$,
$$R_i = \begin{cases} \pi_1^{m-i+1}\circ\cdots\circ\pi^m_i & i\le m\\
\pi^1_{i-m+1} \circ\cdots\circ\pi^m_i & i\ge m . \end{cases}$$
Then the Berenstein-Kirillov map is given by
\be U=R_n\circ\cdots\circ R_1. \label{defU}\ee
Now observe that each $l_{ij}$ is invertible.
Indeed, the inverse of the map \eqref{abcdt} is given by
\begin{equation}\label{abcdtinv}
\begin{pmatrix} a& b\\ c& d\end{pmatrix} \qquad \mapsto \qquad 
\begin{pmatrix} b\wedge c-a  & b\\ c& d-b\vee c \end{pmatrix},
\end{equation}
and the boundary moves $l_{1j}$, $l_{i1}$ are clearly invertible.
It follows that the map $U$ is invertible.
Moreover, $U$ preserves the Lebesgue measure on $\R^{n\times m}$.
The Jacobians of the $l_{ij}$ are clearly almost everywhere equal to $\pm 1$.
Combining this with Corollary \ref{vpt} we conclude that the restriction of $U$ to $\R_+^{n\times m}$
is volume preserving with respect to the Euclidean measure, injective and its image is given by the Euclidean 
set of pairs of generalized Gelfand-Tsetlin patterns with respective heights $m$ and $n$, 
having the same shape in $$C^{(p)}=\{x\in\R_+^p:\ x_1\ge \cdots\ge x_p\}.$$  
Finally, we recall the following straightforward fact.
If we define row and column sums $r_i=\sum_j y_{ij}$ and 
$c_j=\sum_i y_{ij}$, then $\mbox{\small{type}}(S)=r$ and $\mbox{\small{type}}(R)=c$.
Note that this implies, for $\lambda\in\C^m$ and $\nu\in\C^n$, 
\begin{equation}\label{pq'}
\sum_{ij} (\nu_i+\lambda_j) y_{ij} = \sum_i \nu_i r_i+ \sum_j \lambda_j c_j.
\end{equation}
The analogue of the Cauchy-Littlewood identity in this setting 
(cf. Corollary~\ref{cor-bs-gen}) is thus given as follows.
\begin{prop}
Suppose $\lambda\in\C^m$ and $\nu\in\C^n$, where $n\ge m$ and
$\Re (\lambda_i+\nu_j)>0$ for all $i$ and $j$.  Then
\begin{equation}\label{bs-gen-t} 
\int_{C^{(m)} } J_\nu(x) J_\lambda(x) \prod_{i=1}^m dx_i = \prod_{ij}(\nu_i+\lambda_j)^{-1}.
\end{equation} 
\end{prop}
This basic structure has been exploited in the papers~\cite{KJ,br,fr,bp,dw} to study last passage
percolation models with exponential weights, as we shall now explain.  We note that the development
in those papers is via a discrete approximation and as such differs from the present framework, but
the main ideas are the same.  Let $a\in\R^n$ and $b\in\R^m$ such that $a_i+b_j>0$ for all $i$ and $j$.
Consider the measure on input matrices $(y_{ij})\in\R_+^{n\times m}$ defined by
$$\nu_{a,b}(dy)=\prod_{i,j} e^{-(a_i+b_j)y_{ij}} dy_{ij}.$$
From the above, it follows that the push-forward of $\nu_{a,b}$ under the map $U$ is given by
$$\nu_{a,b}\circ U^{-1} (du)=e^{-a\cdot \mbox{\tiny{type}}(S)-b\cdot \mbox{\tiny{type}}(R)} \prod_{i,j} du_{ij}.$$
Now, the variable $u_{nm}$ defined by \eqref{lpt} has the interpretation as a {\em last passage time}
in the percolation model on the lattice with weights given by the $y_{ij}$.  Choosing these weights
at random so that they are independent and exponentially distributed with respective parameters
$a_i+b_j$ corresponds to choosing the input matrix $(y_{ij})$ according to the probability measure
$$\tilde\nu_{a,b}(dy)=\prod_{ij}(a_i+b_j)\nu_{a,b}(dy).$$
From the above, under this probability measure, the law of the random variable $u_{nm}$ is the
same, assuming $n\ge m$, as the first marginal of the probability measure on $C^{(m)}$ defined by
$$\mu_{a,b}(dx) = \prod_{ij}(a_i+b_j) J_a(x) J_b(x) \prod_{i=1}^m dx_i .$$
In other words, for bounded continuous $f$,
$$\int_{\R_+^{n\times m}} f(u_{mn}) \tilde\nu_{a,b}(dy) = \int_{C^{(m)}} f(x_1) \mu_{a,b}(dx) .$$
The probability measures $\mu_{a,b}$ are non-central Laguerre (or complex Wishart) 
ensembles and the integrals \eqref{bs-gen-t} are the corresponding Selberg-type integrals~\cite{KJ,br,fr,bp,dw}.

Similarly, in the symmetric case, one arrives at the interpolating ensembles of Baik and Rains~\cite{br,b}.
These are probability measures on $\R_+^n$ defined for $\alpha\in\R_+^n$ and $\zeta\in\R_+$ by
$$\mu_{\alpha;\zeta}(dx)=\prod_{i<j}(\alpha_i+\alpha_j) J_\alpha(x) \prod_{i=1}^n  (\alpha_i+\zeta) e^{(-1)^i\zeta x_i} dx_i .$$
We note that, in the notation of Section \ref{sym-sec}, as $\epsilon\to 0$, 
$$\nu_{\epsilon\alpha;\epsilon\dipa}\circ\sigma^{-1} (d e^{x/\epsilon}) \Longrightarrow \mu_{\alpha;\zeta}(dx),$$
where $``\Longrightarrow"$ denotes weak convergence of probability measures.
In this setting (see ~\cite{br}) if the input matrix $(y_{ij})\in\R_+^{n\times n}$ is symmetric and chosen according to the
probability measure
$$\prod_{i< j}(\alpha_i+\alpha_j)e^{-( \alpha_i+\alpha_j) y_{ij}} dy_{ij}
\prod_i (\alpha_i+\zeta) e^{-(\alpha_i+\zeta)y_{ii}} dy_{ii}$$
then the last passage time $u_{nn}$ is distributed as the first marginal of $\mu_{\alpha;\zeta}$.

\end{document}